\documentclass[lettersize,journal]{IEEEtran}

\usepackage{amsmath,amsfonts,amssymb,mathtools,bm}
\usepackage{amsthm}
\usepackage{algorithm}
\usepackage{algorithmic}

\usepackage[caption=false,font=footnotesize]{subfig}
\usepackage{array}
\usepackage{booktabs}
\usepackage{cite}
\usepackage{verbatim}
\usepackage{graphicx}
\usepackage{chngcntr}
\usepackage{textcomp}
\usepackage{stfloats}
\usepackage{url}
\usepackage{balance}
\usepackage[dvipsnames]{xcolor}
\usepackage{tikz}
\usepackage{pgfplots}
\pgfplotsset{compat=1.18}
\usetikzlibrary{arrows.meta,positioning,shapes.geometric,fit,calc,decorations.pathreplacing,patterns}

\usepgfplotslibrary{groupplots}

\newif\ifextendedappendices
\extendedappendicestrue
\newif\ifshowrevisionblue
\showrevisionbluetrue

\definecolor{MethodBase}{HTML}{1F77B4}
\definecolor{MethodCentral}{HTML}{D62728}
\definecolor{MethodKKT}{HTML}{2CA02C}
\definecolor{MethodFY}{HTML}{9467BD}
\definecolor{MethodFitz}{HTML}{FF7F0E}
\pgfplotsset{
  compat=1.18,
  every axis/.append style={
    tick label style={font=\scriptsize},
    label style={font=\scriptsize},
    title style={font=\footnotesize},
    legend style={font=\scriptsize},
    grid=major,
    grid style={black!16,densely dotted},
    line width=0.72pt
  },
  baseplot/.style={MethodBase,mark=*,mark size=1.45pt},
  centralplot/.style={MethodCentral,mark=square*,mark size=1.35pt},
  kktplot/.style={MethodKKT,mark=triangle*,mark size=1.55pt},
  fyplot/.style={MethodFY,mark=diamond*,mark size=1.45pt},
  fitzplot/.style={MethodFitz,mark=x,mark size=1.65pt}
}

\definecolor{baseColor}{rgb}{0.121569,0.466667,0.705882}
\definecolor{centColor}{rgb}{0.839216,0.152941,0.156863}
\definecolor{kktColor}{rgb}{0.172549,0.627451,0.172549}
\definecolor{quadFYColor}{rgb}{0.580392,0.403922,0.741177}
\definecolor{gridGray}{gray}{0.690196}
\definecolor{legendBorder}{gray}{0.800000}

\tikzset{
  gridLine/.style={draw=gridGray,line width=0.5bp,dash pattern=on 0.5bp off 0.825bp},
  axisLine/.style={draw=black,line width=0.8bp},
  seriesLine/.style={line width=2bp,line cap=rect,line join=round},
  legendBox/.style={draw=legendBorder,fill=white,line width=0.8bp,rounded corners=1.39bp},
  tickText/.style={font=\fontsize{8bp}{9.6bp}\selectfont,inner sep=0bp,outer sep=0bp,text=black},
  axisLabel/.style={font=\fontsize{11bp}{13.2bp}\selectfont,inner sep=0bp,outer sep=0bp,text=black},
  legendText/.style={font=\fontsize{7bp}{8.4bp}\selectfont,inner sep=0bp,outer sep=0bp,text=black},
}

\newcommand{\markercircle}[3]{\path[draw=#1,fill=#1,line width=1bp] (#2bp,#3bp) circle[radius=1.75bp]}
\newcommand{\markersquare}[3]{\path[draw=#1,fill=#1,line width=1bp] ($ (#2bp,#3bp)+(-1.75bp,-1.75bp) $) rectangle ($ (#2bp,#3bp)+(1.75bp,1.75bp) $)}
\newcommand{\markertriangle}[3]{\path[draw=#1,fill=#1,line width=1bp] (#2bp,{#3bp+1.75bp}) -- ({#2bp-1.75bp},{#3bp-1.75bp}) -- ({#2bp+1.75bp},{#3bp-1.75bp}) -- cycle}
\newcommand{\markerdiamond}[3]{\path[draw=#1,fill=#1,line width=1bp] (#2bp,{#3bp-2.474874bp}) -- ({#2bp+2.474874bp},#3bp) -- (#2bp,{#3bp+2.474874bp}) -- ({#2bp-2.474874bp},#3bp) -- cycle}
\newcommand{\markercross}[3]{\draw[#1,line width=1bp,line cap=round] ({#2bp-1.75bp},{#3bp-1.75bp}) -- ({#2bp+1.75bp},{#3bp+1.75bp}) ({#2bp-1.75bp},{#3bp+1.75bp}) -- ({#2bp+1.75bp},{#3bp-1.75bp})}

\newcommand{\FigureCaseC}{%
\begin{tikzpicture}[x=1bp,y=1bp]
  \path[use as bounding box] (0bp,0bp) rectangle (510.56bp,229.814bp);
  \fill[white] (0bp,0bp) rectangle (510.56bp,229.814bp);
  \fill[white] (46.669bp,36.591bp) rectangle (503.360bp,206.740bp);

  \draw[gridLine] (67.427bp,36.591bp) -- (67.427bp,206.740bp);
  \draw[gridLine] (142.914bp,36.591bp) -- (142.914bp,206.740bp);
  \draw[gridLine] (218.400bp,36.591bp) -- (218.400bp,206.740bp);
  \draw[gridLine] (293.886bp,36.591bp) -- (293.886bp,206.740bp);
  \draw[gridLine] (369.372bp,36.591bp) -- (369.372bp,206.740bp);
  \draw[gridLine] (444.858bp,36.591bp) -- (444.858bp,206.740bp);
  \draw[gridLine] (46.669bp,44.325bp) -- (503.360bp,44.325bp);
  \draw[gridLine] (46.669bp,64.479bp) -- (503.360bp,64.479bp);
  \draw[gridLine] (46.669bp,84.633bp) -- (503.360bp,84.633bp);
  \draw[gridLine] (46.669bp,104.787bp) -- (503.360bp,104.787bp);
  \draw[gridLine] (46.669bp,124.941bp) -- (503.360bp,124.941bp);
  \draw[gridLine] (46.669bp,145.095bp) -- (503.360bp,145.095bp);
  \draw[gridLine] (46.669bp,165.249bp) -- (503.360bp,165.249bp);
  \draw[gridLine] (46.669bp,185.403bp) -- (503.360bp,185.403bp);
  \draw[gridLine] (46.669bp,205.557bp) -- (503.360bp,205.557bp);

  \begin{scope}
    \clip (46.669bp,36.591bp) rectangle (503.360bp,206.740bp);
    \draw[seriesLine,baseColor] (67.427bp,114.601bp) -- (105.171bp,199.005bp) -- (142.914bp,109.340bp) -- (180.657bp,119.990bp) -- (218.400bp,118.253bp) -- (256.143bp,153.352bp) -- (293.886bp,133.477bp) -- (331.629bp,44.325bp) -- (369.372bp,44.325bp) -- (407.115bp,44.325bp) -- (444.858bp,44.325bp) -- (482.602bp,44.325bp);
    \markercircle{baseColor}{67.427}{114.601};
    \markercircle{baseColor}{105.171}{199.005};
    \markercircle{baseColor}{142.914}{109.340};
    \markercircle{baseColor}{180.657}{119.990};
    \markercircle{baseColor}{218.400}{118.253};
    \markercircle{baseColor}{256.143}{153.352};
    \markercircle{baseColor}{293.886}{133.477};
    \markercircle{baseColor}{331.629}{44.325};
    \markercircle{baseColor}{369.372}{44.325};
    \markercircle{baseColor}{407.115}{44.325};
    \markercircle{baseColor}{444.858}{44.325};
    \markercircle{baseColor}{482.602}{44.325};
    \draw[seriesLine,centColor] (67.427bp,46.225bp) -- (105.171bp,46.225bp) -- (142.914bp,46.225bp) -- (180.657bp,46.225bp) -- (218.400bp,46.225bp) -- (256.143bp,46.225bp) -- (293.886bp,46.225bp) -- (331.629bp,46.225bp) -- (369.372bp,44.325bp) -- (407.115bp,44.325bp) -- (444.858bp,44.325bp) -- (482.602bp,44.325bp);
    \markersquare{centColor}{67.427}{46.225};
    \markersquare{centColor}{105.171}{46.225};
    \markersquare{centColor}{142.914}{46.225};
    \markersquare{centColor}{180.657}{46.225};
    \markersquare{centColor}{218.400}{46.225};
    \markersquare{centColor}{256.143}{46.225};
    \markersquare{centColor}{293.886}{46.225};
    \markersquare{centColor}{331.629}{46.225};
    \markersquare{centColor}{369.372}{44.325};
    \markersquare{centColor}{407.115}{44.325};
    \markersquare{centColor}{444.858}{44.325};
    \markersquare{centColor}{482.602}{44.325};
    \draw[seriesLine,kktColor] (67.427bp,48.244bp) -- (105.171bp,54.331bp) -- (142.914bp,54.356bp) -- (180.657bp,48.184bp) -- (218.400bp,56.933bp) -- (256.143bp,44.325bp) -- (293.886bp,44.325bp) -- (331.629bp,51.586bp) -- (369.372bp,44.325bp) -- (407.115bp,44.325bp) -- (444.858bp,44.325bp) -- (482.602bp,44.325bp);
    \markertriangle{kktColor}{67.427}{48.244};
    \markertriangle{kktColor}{105.171}{54.331};
    \markertriangle{kktColor}{142.914}{54.356};
    \markertriangle{kktColor}{180.657}{48.184};
    \markertriangle{kktColor}{218.400}{56.933};
    \markertriangle{kktColor}{256.143}{44.325};
    \markertriangle{kktColor}{293.886}{44.325};
    \markertriangle{kktColor}{331.629}{51.586};
    \markertriangle{kktColor}{369.372}{44.325};
    \markertriangle{kktColor}{407.115}{44.325};
    \markertriangle{kktColor}{444.858}{44.325};
    \markertriangle{kktColor}{482.602}{44.325};
    \draw[seriesLine,quadFYColor] (67.427bp,48.887bp) -- (105.171bp,54.328bp) -- (142.914bp,49.710bp) -- (180.657bp,44.325bp) -- (218.400bp,63.213bp) -- (256.143bp,51.571bp) -- (293.886bp,52.532bp) -- (331.629bp,51.584bp) -- (369.372bp,44.325bp) -- (407.115bp,44.325bp) -- (444.858bp,44.325bp) -- (482.602bp,44.325bp);
    \markerdiamond{quadFYColor}{67.427}{48.887};
    \markerdiamond{quadFYColor}{105.171}{54.328};
    \markerdiamond{quadFYColor}{142.914}{49.710};
    \markerdiamond{quadFYColor}{180.657}{44.325};
    \markerdiamond{quadFYColor}{218.400}{63.213};
    \markerdiamond{quadFYColor}{256.143}{51.571};
    \markerdiamond{quadFYColor}{293.886}{52.532};
    \markerdiamond{quadFYColor}{331.629}{51.584};
    \markerdiamond{quadFYColor}{369.372}{44.325};
    \markerdiamond{quadFYColor}{407.115}{44.325};
    \markerdiamond{quadFYColor}{444.858}{44.325};
    \markerdiamond{quadFYColor}{482.602}{44.325};
  \end{scope}

  \draw[axisLine] (46.669bp,36.591bp) rectangle (503.360bp,206.740bp);
  \draw[axisLine] (67.427bp,36.591bp) -- (67.427bp,33.101bp);
  \node[tickText,anchor=base west] at (64.881bp,23.513bp) {0};
  \draw[axisLine] (142.914bp,36.591bp) -- (142.914bp,33.101bp);
  \node[tickText,anchor=base west] at (140.367bp,23.513bp) {2};
  \draw[axisLine] (218.400bp,36.591bp) -- (218.400bp,33.101bp);
  \node[tickText,anchor=base west] at (215.853bp,23.513bp) {4};
  \draw[axisLine] (293.886bp,36.591bp) -- (293.886bp,33.101bp);
  \node[tickText,anchor=base west] at (291.339bp,23.513bp) {6};
  \draw[axisLine] (369.372bp,36.591bp) -- (369.372bp,33.101bp);
  \node[tickText,anchor=base west] at (366.825bp,23.513bp) {8};
  \draw[axisLine] (444.858bp,36.591bp) -- (444.858bp,33.101bp);
  \node[tickText,anchor=base west] at (439.765bp,23.513bp) {10};
  \draw[axisLine] (46.669bp,44.325bp) -- (43.179bp,44.325bp);
  \node[tickText,anchor=base west] at (21.841bp,41.286bp) {0.00};
  \draw[axisLine] (46.669bp,64.479bp) -- (43.179bp,64.479bp);
  \node[tickText,anchor=base west] at (21.841bp,61.440bp) {0.25};
  \draw[axisLine] (46.669bp,84.633bp) -- (43.179bp,84.633bp);
  \node[tickText,anchor=base west] at (21.841bp,81.594bp) {0.50};
  \draw[axisLine] (46.669bp,104.787bp) -- (43.179bp,104.787bp);
  \node[tickText,anchor=base west] at (21.841bp,101.748bp) {0.75};
  \draw[axisLine] (46.669bp,124.941bp) -- (43.179bp,124.941bp);
  \node[tickText,anchor=base west] at (21.841bp,121.902bp) {1.00};
  \draw[axisLine] (46.669bp,145.095bp) -- (43.179bp,145.095bp);
  \node[tickText,anchor=base west] at (21.841bp,142.056bp) {1.25};
  \draw[axisLine] (46.669bp,165.249bp) -- (43.179bp,165.249bp);
  \node[tickText,anchor=base west] at (21.841bp,162.210bp) {1.50};
  \draw[axisLine] (46.669bp,185.403bp) -- (43.179bp,185.403bp);
  \node[tickText,anchor=base west] at (21.841bp,182.364bp) {1.75};
  \draw[axisLine] (46.669bp,205.557bp) -- (43.179bp,205.557bp);
  \node[tickText,anchor=base west] at (21.841bp,202.518bp) {2.00};
  \node[axisLabel,anchor=base west] at (225.249bp,9.482bp) {Settlement period};
  \node[axisLabel,rotate=90,anchor=base west] at (15.559bp,74.462bp) {Congestion (MW)};

  \draw[legendBox] (376.383bp,170.342bp) rectangle (499.863bp,203.232bp);
  \draw[seriesLine,baseColor] (379.173bp,197.577bp) -- (393.173bp,197.577bp);
  \markercircle{baseColor}{386.173}{197.577};
  \node[legendText,anchor=base west] at (398.773bp,195.127bp) {Base};
  \draw[seriesLine,centColor] (379.173bp,187.311bp) -- (393.173bp,187.311bp);
  \markersquare{centColor}{386.173}{187.311};
  \node[legendText,anchor=base west] at (398.773bp,184.861bp) {Cent.};
  \draw[seriesLine,kktColor] (379.173bp,177.045bp) -- (393.173bp,177.045bp);
  \markertriangle{kktColor}{386.173}{177.045};
  \node[legendText,anchor=base west] at (398.773bp,174.595bp) {KKT/MILP};
  \draw[seriesLine,quadFYColor] (444.335bp,197.577bp) -- (458.335bp,197.577bp);
  \markerdiamond{quadFYColor}{451.335}{197.577};
  \node[legendText,anchor=base west] at (463.935bp,195.127bp) {Quad-FY};
\end{tikzpicture}%
}

\newcommand{\FigureCaseRone}{%
\begin{tikzpicture}[x=1bp,y=1bp]
  \path[use as bounding box] (0bp,0bp) rectangle (510.56bp,229.814bp);
  \fill[white] (0bp,0bp) rectangle (510.56bp,229.814bp);
  \fill[white] (41.575bp,36.591bp) rectangle (503.345bp,206.740bp);

  \draw[gridLine] (62.565bp,36.591bp) -- (62.565bp,206.740bp);
  \draw[gridLine] (138.890bp,36.591bp) -- (138.890bp,206.740bp);
  \draw[gridLine] (215.216bp,36.591bp) -- (215.216bp,206.740bp);
  \draw[gridLine] (291.541bp,36.591bp) -- (291.541bp,206.740bp);
  \draw[gridLine] (367.867bp,36.591bp) -- (367.867bp,206.740bp);
  \draw[gridLine] (444.193bp,36.591bp) -- (444.193bp,206.740bp);
  \draw[gridLine] (41.575bp,44.325bp) -- (503.345bp,44.325bp);
  \draw[gridLine] (41.575bp,76.201bp) -- (503.345bp,76.201bp);
  \draw[gridLine] (41.575bp,108.078bp) -- (503.345bp,108.078bp);
  \draw[gridLine] (41.575bp,139.955bp) -- (503.345bp,139.955bp);
  \draw[gridLine] (41.575bp,171.832bp) -- (503.345bp,171.832bp);
  \draw[gridLine] (41.575bp,203.708bp) -- (503.345bp,203.708bp);

  \begin{scope}
    \clip (41.575bp,36.591bp) rectangle (503.345bp,206.740bp);
    \draw[seriesLine,baseColor] (62.565bp,48.616bp) -- (100.727bp,95.218bp) -- (138.890bp,44.325bp) -- (177.053bp,44.325bp) -- (215.216bp,44.325bp) -- (253.379bp,44.325bp) -- (291.541bp,62.840bp) -- (329.704bp,44.325bp) -- (367.867bp,44.325bp) -- (406.030bp,199.005bp) -- (444.193bp,44.325bp) -- (482.355bp,44.325bp);
    \markercircle{baseColor}{62.565}{48.616};
    \markercircle{baseColor}{100.727}{95.218};
    \markercircle{baseColor}{138.890}{44.325};
    \markercircle{baseColor}{177.053}{44.325};
    \markercircle{baseColor}{215.216}{44.325};
    \markercircle{baseColor}{253.379}{44.325};
    \markercircle{baseColor}{291.541}{62.840};
    \markercircle{baseColor}{329.704}{44.325};
    \markercircle{baseColor}{367.867}{44.325};
    \markercircle{baseColor}{406.030}{199.005};
    \markercircle{baseColor}{444.193}{44.325};
    \markercircle{baseColor}{482.355}{44.325};
    \draw[seriesLine,quadFYColor] (62.565bp,44.325bp) -- (100.727bp,44.325bp) -- (138.890bp,44.325bp) -- (177.053bp,44.325bp) -- (215.216bp,44.325bp) -- (253.379bp,44.325bp) -- (291.541bp,44.325bp) -- (329.704bp,44.325bp) -- (367.867bp,44.325bp) -- (406.030bp,182.221bp) -- (444.193bp,44.325bp) -- (482.355bp,44.325bp);
    \markerdiamond{quadFYColor}{62.565}{44.325};
    \markerdiamond{quadFYColor}{100.727}{44.325};
    \markerdiamond{quadFYColor}{138.890}{44.325};
    \markerdiamond{quadFYColor}{177.053}{44.325};
    \markerdiamond{quadFYColor}{215.216}{44.325};
    \markerdiamond{quadFYColor}{253.379}{44.325};
    \markerdiamond{quadFYColor}{291.541}{44.325};
    \markerdiamond{quadFYColor}{329.704}{44.325};
    \markerdiamond{quadFYColor}{367.867}{44.325};
    \markerdiamond{quadFYColor}{406.030}{182.221};
    \markerdiamond{quadFYColor}{444.193}{44.325};
    \markerdiamond{quadFYColor}{482.355}{44.325};
    \draw[seriesLine,kktColor] (62.565bp,44.325bp) -- (100.727bp,44.325bp) -- (138.890bp,44.325bp) -- (177.053bp,44.325bp) -- (215.216bp,44.325bp) -- (253.379bp,44.325bp) -- (291.541bp,44.325bp) -- (329.704bp,44.325bp) -- (367.867bp,44.325bp) -- (406.030bp,182.231bp) -- (444.193bp,44.325bp) -- (482.355bp,44.325bp);
    \markertriangle{kktColor}{62.565}{44.325};
    \markertriangle{kktColor}{100.727}{44.325};
    \markertriangle{kktColor}{138.890}{44.325};
    \markertriangle{kktColor}{177.053}{44.325};
    \markertriangle{kktColor}{215.216}{44.325};
    \markertriangle{kktColor}{253.379}{44.325};
    \markertriangle{kktColor}{291.541}{44.325};
    \markertriangle{kktColor}{329.704}{44.325};
    \markertriangle{kktColor}{367.867}{44.325};
    \markertriangle{kktColor}{406.030}{182.231};
    \markertriangle{kktColor}{444.193}{44.325};
    \markertriangle{kktColor}{482.355}{44.325};
    \draw[seriesLine,centColor] (62.565bp,44.325bp) -- (100.727bp,44.325bp) -- (138.890bp,44.325bp) -- (177.053bp,44.325bp) -- (215.216bp,44.325bp) -- (253.379bp,44.325bp) -- (291.541bp,44.325bp) -- (329.704bp,44.325bp) -- (367.867bp,44.325bp) -- (406.030bp,182.231bp) -- (444.193bp,44.325bp) -- (482.355bp,44.325bp);
    \markersquare{centColor}{62.565}{44.325};
    \markersquare{centColor}{100.727}{44.325};
    \markersquare{centColor}{138.890}{44.325};
    \markersquare{centColor}{177.053}{44.325};
    \markersquare{centColor}{215.216}{44.325};
    \markersquare{centColor}{253.379}{44.325};
    \markersquare{centColor}{291.541}{44.325};
    \markersquare{centColor}{329.704}{44.325};
    \markersquare{centColor}{367.867}{44.325};
    \markersquare{centColor}{406.030}{182.231};
    \markersquare{centColor}{444.193}{44.325};
    \markersquare{centColor}{482.355}{44.325};
  \end{scope}

  \draw[axisLine] (41.575bp,36.591bp) rectangle (503.345bp,206.740bp);
  \draw[axisLine] (62.565bp,36.591bp) -- (62.565bp,33.101bp);
  \node[tickText,anchor=base west] at (60.018bp,23.513bp) {0};
  \draw[axisLine] (138.890bp,36.591bp) -- (138.890bp,33.101bp);
  \node[tickText,anchor=base west] at (136.343bp,23.513bp) {2};
  \draw[axisLine] (215.216bp,36.591bp) -- (215.216bp,33.101bp);
  \node[tickText,anchor=base west] at (212.669bp,23.513bp) {4};
  \draw[axisLine] (291.541bp,36.591bp) -- (291.541bp,33.101bp);
  \node[tickText,anchor=base west] at (288.995bp,23.513bp) {6};
  \draw[axisLine] (367.867bp,36.591bp) -- (367.867bp,33.101bp);
  \node[tickText,anchor=base west] at (365.320bp,23.513bp) {8};
  \draw[axisLine] (444.193bp,36.591bp) -- (444.193bp,33.101bp);
  \node[tickText,anchor=base west] at (439.099bp,23.513bp) {10};
  \draw[axisLine] (41.575bp,44.325bp) -- (38.085bp,44.325bp);
  \node[tickText,anchor=base west] at (21.841bp,41.286bp) {0.0};
  \draw[axisLine] (41.575bp,76.201bp) -- (38.085bp,76.201bp);
  \node[tickText,anchor=base west] at (21.841bp,73.162bp) {0.2};
  \draw[axisLine] (41.575bp,108.078bp) -- (38.085bp,108.078bp);
  \node[tickText,anchor=base west] at (21.841bp,105.039bp) {0.4};
  \draw[axisLine] (41.575bp,139.955bp) -- (38.085bp,139.955bp);
  \node[tickText,anchor=base west] at (21.841bp,136.916bp) {0.6};
  \draw[axisLine] (41.575bp,171.832bp) -- (38.085bp,171.832bp);
  \node[tickText,anchor=base west] at (21.841bp,168.793bp) {0.8};
  \draw[axisLine] (41.575bp,203.708bp) -- (38.085bp,203.708bp);
  \node[tickText,anchor=base west] at (21.841bp,200.669bp) {1.0};
  \node[axisLabel,anchor=base west] at (222.694bp,9.482bp) {Settlement period};
  \node[axisLabel,rotate=90,anchor=base west] at (15.559bp,74.462bp) {Congestion (MW)};

  \draw[legendBox] (45.085bp,170.342bp) rectangle (168.565bp,203.232bp);
  \draw[seriesLine,baseColor] (47.875bp,197.577bp) -- (61.875bp,197.577bp);
  \markercircle{baseColor}{54.875}{197.577};
  \node[legendText,anchor=base west] at (67.475bp,195.127bp) {Base};
  \draw[seriesLine,quadFYColor] (47.875bp,187.311bp) -- (61.875bp,187.311bp);
  \markerdiamond{quadFYColor}{54.875}{187.311};
  \node[legendText,anchor=base west] at (67.475bp,184.861bp) {Quad-FY};
  \draw[seriesLine,kktColor] (114.600bp,197.577bp) -- (128.600bp,197.577bp);
  \markertriangle{kktColor}{121.600}{197.577};
  \node[legendText,anchor=base west] at (134.200bp,195.127bp) {KKT/MILP};
  \draw[seriesLine,centColor] (114.600bp,187.311bp) -- (128.600bp,187.311bp);
  \markersquare{centColor}{121.600}{187.311};
  \node[legendText,anchor=base west] at (134.200bp,184.861bp) {Cent.};
\end{tikzpicture}%
}

\newcommand{\FigureCaseA}{%
\begin{tikzpicture}[x=1bp,y=1bp]
  \path[use as bounding box] (0bp,0bp) rectangle (510.56bp,229.814bp);
  \fill[white] (0bp,0bp) rectangle (510.56bp,229.814bp);
  \fill[white] (41.575bp,36.591bp) rectangle (503.345bp,206.740bp);

  \draw[gridLine] (62.565bp,36.591bp) -- (62.565bp,206.740bp);
  \draw[gridLine] (138.890bp,36.591bp) -- (138.890bp,206.740bp);
  \draw[gridLine] (215.216bp,36.591bp) -- (215.216bp,206.740bp);
  \draw[gridLine] (291.541bp,36.591bp) -- (291.541bp,206.740bp);
  \draw[gridLine] (367.867bp,36.591bp) -- (367.867bp,206.740bp);
  \draw[gridLine] (444.193bp,36.591bp) -- (444.193bp,206.740bp);
  \draw[gridLine] (41.575bp,44.325bp) -- (503.345bp,44.325bp);
  \draw[gridLine] (41.575bp,63.926bp) -- (503.345bp,63.926bp);
  \draw[gridLine] (41.575bp,83.528bp) -- (503.345bp,83.528bp);
  \draw[gridLine] (41.575bp,103.129bp) -- (503.345bp,103.129bp);
  \draw[gridLine] (41.575bp,122.730bp) -- (503.345bp,122.730bp);
  \draw[gridLine] (41.575bp,142.332bp) -- (503.345bp,142.332bp);
  \draw[gridLine] (41.575bp,161.933bp) -- (503.345bp,161.933bp);
  \draw[gridLine] (41.575bp,181.535bp) -- (503.345bp,181.535bp);
  \draw[gridLine] (41.575bp,201.136bp) -- (503.345bp,201.136bp);

  \begin{scope}
    \clip (41.575bp,36.591bp) rectangle (503.345bp,206.740bp);
    \draw[seriesLine,baseColor] (62.565bp,134.824bp) -- (100.727bp,199.005bp) -- (138.890bp,101.043bp) -- (177.053bp,114.408bp) -- (215.216bp,110.558bp) -- (253.379bp,159.346bp) -- (291.541bp,157.941bp) -- (329.704bp,62.765bp) -- (367.867bp,49.280bp) -- (406.030bp,54.019bp) -- (444.193bp,44.325bp) -- (482.355bp,44.325bp);
    \markercircle{baseColor}{62.565}{134.824};
    \markercircle{baseColor}{100.727}{199.005};
    \markercircle{baseColor}{138.890}{101.043};
    \markercircle{baseColor}{177.053}{114.408};
    \markercircle{baseColor}{215.216}{110.558};
    \markercircle{baseColor}{253.379}{159.346};
    \markercircle{baseColor}{291.541}{157.941};
    \markercircle{baseColor}{329.704}{62.765};
    \markercircle{baseColor}{367.867}{49.280};
    \markercircle{baseColor}{406.030}{54.019};
    \markercircle{baseColor}{444.193}{44.325};
    \markercircle{baseColor}{482.355}{44.325};
    \draw[seriesLine,centColor] (62.565bp,44.325bp) -- (100.727bp,44.325bp) -- (138.890bp,44.325bp) -- (177.053bp,44.325bp) -- (215.216bp,44.325bp) -- (253.379bp,44.325bp) -- (291.541bp,44.325bp) -- (329.704bp,44.325bp) -- (367.867bp,44.325bp) -- (406.030bp,44.325bp) -- (444.193bp,44.325bp) -- (482.355bp,44.325bp);
    \markersquare{centColor}{62.565}{44.325};
    \markersquare{centColor}{100.727}{44.325};
    \markersquare{centColor}{138.890}{44.325};
    \markersquare{centColor}{177.053}{44.325};
    \markersquare{centColor}{215.216}{44.325};
    \markersquare{centColor}{253.379}{44.325};
    \markersquare{centColor}{291.541}{44.325};
    \markersquare{centColor}{329.704}{44.325};
    \markersquare{centColor}{367.867}{44.325};
    \markersquare{centColor}{406.030}{44.325};
    \markersquare{centColor}{444.193}{44.325};
    \markersquare{centColor}{482.355}{44.325};
    \draw[seriesLine,kktColor] (62.565bp,46.981bp) -- (100.727bp,44.325bp) -- (138.890bp,45.193bp) -- (177.053bp,44.325bp) -- (215.216bp,44.325bp) -- (253.379bp,44.325bp) -- (291.541bp,47.292bp) -- (329.704bp,44.678bp) -- (367.867bp,44.325bp) -- (406.030bp,44.325bp) -- (444.193bp,45.793bp) -- (482.355bp,44.325bp);
    \markertriangle{kktColor}{62.565}{46.981};
    \markertriangle{kktColor}{100.727}{44.325};
    \markertriangle{kktColor}{138.890}{45.193};
    \markertriangle{kktColor}{177.053}{44.325};
    \markertriangle{kktColor}{215.216}{44.325};
    \markertriangle{kktColor}{253.379}{44.325};
    \markertriangle{kktColor}{291.541}{47.292};
    \markertriangle{kktColor}{329.704}{44.678};
    \markertriangle{kktColor}{367.867}{44.325};
    \markertriangle{kktColor}{406.030}{44.325};
    \markertriangle{kktColor}{444.193}{45.793};
    \markertriangle{kktColor}{482.355}{44.325};
    \draw[seriesLine,quadFYColor] (62.565bp,44.325bp) -- (100.727bp,44.325bp) -- (138.890bp,45.193bp) -- (177.053bp,51.620bp) -- (215.216bp,44.325bp) -- (253.379bp,44.325bp) -- (291.541bp,46.034bp) -- (329.704bp,57.559bp) -- (367.867bp,44.325bp) -- (406.030bp,44.325bp) -- (444.193bp,44.325bp) -- (482.355bp,44.325bp);
    \markerdiamond{quadFYColor}{62.565}{44.325};
    \markerdiamond{quadFYColor}{100.727}{44.325};
    \markerdiamond{quadFYColor}{138.890}{45.193};
    \markerdiamond{quadFYColor}{177.053}{51.620};
    \markerdiamond{quadFYColor}{215.216}{44.325};
    \markerdiamond{quadFYColor}{253.379}{44.325};
    \markerdiamond{quadFYColor}{291.541}{46.034};
    \markerdiamond{quadFYColor}{329.704}{57.559};
    \markerdiamond{quadFYColor}{367.867}{44.325};
    \markerdiamond{quadFYColor}{406.030}{44.325};
    \markerdiamond{quadFYColor}{444.193}{44.325};
    \markerdiamond{quadFYColor}{482.355}{44.325};
  \end{scope}

  \draw[axisLine] (41.575bp,36.591bp) rectangle (503.345bp,206.740bp);
  \draw[axisLine] (62.565bp,36.591bp) -- (62.565bp,33.101bp);
  \node[tickText,anchor=base west] at (60.018bp,23.513bp) {0};
  \draw[axisLine] (138.890bp,36.591bp) -- (138.890bp,33.101bp);
  \node[tickText,anchor=base west] at (136.343bp,23.513bp) {2};
  \draw[axisLine] (215.216bp,36.591bp) -- (215.216bp,33.101bp);
  \node[tickText,anchor=base west] at (212.669bp,23.513bp) {4};
  \draw[axisLine] (291.541bp,36.591bp) -- (291.541bp,33.101bp);
  \node[tickText,anchor=base west] at (288.995bp,23.513bp) {6};
  \draw[axisLine] (367.867bp,36.591bp) -- (367.867bp,33.101bp);
  \node[tickText,anchor=base west] at (365.320bp,23.513bp) {8};
  \draw[axisLine] (444.193bp,36.591bp) -- (444.193bp,33.101bp);
  \node[tickText,anchor=base west] at (439.099bp,23.513bp) {10};
  \draw[axisLine] (41.575bp,44.325bp) -- (38.085bp,44.325bp);
  \node[tickText,anchor=base west] at (21.841bp,41.286bp) {0.0};
  \draw[axisLine] (41.575bp,63.926bp) -- (38.085bp,63.926bp);
  \node[tickText,anchor=base west] at (21.841bp,60.887bp) {0.1};
  \draw[axisLine] (41.575bp,83.528bp) -- (38.085bp,83.528bp);
  \node[tickText,anchor=base west] at (21.841bp,80.488bp) {0.2};
  \draw[axisLine] (41.575bp,103.129bp) -- (38.085bp,103.129bp);
  \node[tickText,anchor=base west] at (21.841bp,100.090bp) {0.3};
  \draw[axisLine] (41.575bp,122.730bp) -- (38.085bp,122.730bp);
  \node[tickText,anchor=base west] at (21.841bp,119.691bp) {0.4};
  \draw[axisLine] (41.575bp,142.332bp) -- (38.085bp,142.332bp);
  \node[tickText,anchor=base west] at (21.841bp,139.293bp) {0.5};
  \draw[axisLine] (41.575bp,161.933bp) -- (38.085bp,161.933bp);
  \node[tickText,anchor=base west] at (21.841bp,158.894bp) {0.6};
  \draw[axisLine] (41.575bp,181.535bp) -- (38.085bp,181.535bp);
  \node[tickText,anchor=base west] at (21.841bp,178.496bp) {0.7};
  \draw[axisLine] (41.575bp,201.136bp) -- (38.085bp,201.136bp);
  \node[tickText,anchor=base west] at (21.841bp,198.097bp) {0.8};
  \node[axisLabel,anchor=base west] at (222.694bp,9.482bp) {Settlement period};
  \node[axisLabel,rotate=90,anchor=base west] at (15.559bp,74.462bp) {Congestion (MW)};

  \draw[legendBox] (376.367bp,170.342bp) rectangle (499.847bp,203.232bp);
  \draw[seriesLine,baseColor] (379.157bp,197.577bp) -- (393.157bp,197.577bp);
  \markercircle{baseColor}{386.157}{197.577};
  \node[legendText,anchor=base west] at (398.757bp,195.127bp) {Base};
  \draw[seriesLine,centColor] (379.157bp,187.311bp) -- (393.157bp,187.311bp);
  \markersquare{centColor}{386.157}{187.311};
  \node[legendText,anchor=base west] at (398.757bp,184.861bp) {Cent.};
  \draw[seriesLine,kktColor] (379.157bp,177.045bp) -- (393.157bp,177.045bp);
  \markertriangle{kktColor}{386.157}{177.045};
  \node[legendText,anchor=base west] at (398.757bp,174.595bp) {KKT/MILP};
  \draw[seriesLine,quadFYColor] (444.320bp,187.311bp) -- (458.320bp,187.311bp);
  \markerdiamond{quadFYColor}{451.320}{187.311};
  \node[legendText,anchor=base west] at (463.920bp,184.861bp) {Quad-FY};
\end{tikzpicture}%
}

\newcommand{\FigureCaseB}{%
\begin{tikzpicture}[x=1bp,y=1bp]
  \path[use as bounding box] (0bp,0bp) rectangle (510.56bp,229.814bp);
  \fill[white] (0bp,0bp) rectangle (510.56bp,229.814bp);
  \fill[white] (41.575bp,36.591bp) rectangle (503.345bp,206.740bp);

  \draw[gridLine] (62.565bp,36.591bp) -- (62.565bp,206.740bp);
  \draw[gridLine] (138.890bp,36.591bp) -- (138.890bp,206.740bp);
  \draw[gridLine] (215.216bp,36.591bp) -- (215.216bp,206.740bp);
  \draw[gridLine] (291.541bp,36.591bp) -- (291.541bp,206.740bp);
  \draw[gridLine] (367.867bp,36.591bp) -- (367.867bp,206.740bp);
  \draw[gridLine] (444.193bp,36.591bp) -- (444.193bp,206.740bp);
  \draw[gridLine] (41.575bp,42.723bp) -- (503.345bp,42.723bp);
  \draw[gridLine] (41.575bp,69.836bp) -- (503.345bp,69.836bp);
  \draw[gridLine] (41.575bp,96.950bp) -- (503.345bp,96.950bp);
  \draw[gridLine] (41.575bp,124.064bp) -- (503.345bp,124.064bp);
  \draw[gridLine] (41.575bp,151.178bp) -- (503.345bp,151.178bp);
  \draw[gridLine] (41.575bp,178.292bp) -- (503.345bp,178.292bp);
  \draw[gridLine] (41.575bp,205.406bp) -- (503.345bp,205.406bp);

  \begin{scope}
    \clip (41.575bp,36.591bp) rectangle (503.345bp,206.740bp);
    \draw[seriesLine,baseColor] (62.565bp,156.142bp) -- (100.727bp,190.771bp) -- (138.890bp,145.592bp) -- (177.053bp,156.284bp) -- (215.216bp,157.308bp) -- (253.379bp,183.960bp) -- (291.541bp,199.005bp) -- (329.704bp,119.198bp) -- (367.867bp,88.521bp) -- (406.030bp,162.001bp) -- (444.193bp,121.605bp) -- (482.355bp,89.340bp);
    \markercircle{baseColor}{62.565}{156.142};
    \markercircle{baseColor}{100.727}{190.771};
    \markercircle{baseColor}{138.890}{145.592};
    \markercircle{baseColor}{177.053}{156.284};
    \markercircle{baseColor}{215.216}{157.308};
    \markercircle{baseColor}{253.379}{183.960};
    \markercircle{baseColor}{291.541}{199.005};
    \markercircle{baseColor}{329.704}{119.198};
    \markercircle{baseColor}{367.867}{88.521};
    \markercircle{baseColor}{406.030}{162.001};
    \markercircle{baseColor}{444.193}{121.605};
    \markercircle{baseColor}{482.355}{89.340};
    \draw[seriesLine,centColor] (62.565bp,86.763bp) -- (100.727bp,86.763bp) -- (138.890bp,86.763bp) -- (177.053bp,86.763bp) -- (215.216bp,86.763bp) -- (253.379bp,86.763bp) -- (291.541bp,86.763bp) -- (329.704bp,86.763bp) -- (367.867bp,86.763bp) -- (406.030bp,86.763bp) -- (444.193bp,72.456bp) -- (482.355bp,52.423bp);
    \markersquare{centColor}{62.565}{86.763};
    \markersquare{centColor}{100.727}{86.763};
    \markersquare{centColor}{138.890}{86.763};
    \markersquare{centColor}{177.053}{86.763};
    \markersquare{centColor}{215.216}{86.763};
    \markersquare{centColor}{253.379}{86.763};
    \markersquare{centColor}{291.541}{86.763};
    \markersquare{centColor}{329.704}{86.763};
    \markersquare{centColor}{367.867}{86.763};
    \markersquare{centColor}{406.030}{86.763};
    \markersquare{centColor}{444.193}{72.456};
    \markersquare{centColor}{482.355}{52.423};
    \draw[seriesLine,quadFYColor] (62.565bp,92.673bp) -- (100.727bp,82.683bp) -- (138.890bp,86.010bp) -- (177.053bp,93.051bp) -- (215.216bp,89.150bp) -- (253.379bp,101.193bp) -- (291.541bp,68.865bp) -- (329.704bp,94.828bp) -- (367.867bp,93.112bp) -- (406.030bp,79.121bp) -- (444.193bp,72.454bp) -- (482.355bp,57.803bp);
    \markerdiamond{quadFYColor}{62.565}{92.673};
    \markerdiamond{quadFYColor}{100.727}{82.683};
    \markerdiamond{quadFYColor}{138.890}{86.010};
    \markerdiamond{quadFYColor}{177.053}{93.051};
    \markerdiamond{quadFYColor}{215.216}{89.150};
    \markerdiamond{quadFYColor}{253.379}{101.193};
    \markerdiamond{quadFYColor}{291.541}{68.865};
    \markerdiamond{quadFYColor}{329.704}{94.828};
    \markerdiamond{quadFYColor}{367.867}{93.112};
    \markerdiamond{quadFYColor}{406.030}{79.121};
    \markerdiamond{quadFYColor}{444.193}{72.454};
    \markerdiamond{quadFYColor}{482.355}{57.803};
    \draw[seriesLine,kktColor] (62.565bp,92.679bp) -- (100.727bp,82.685bp) -- (138.890bp,86.012bp) -- (177.053bp,89.738bp) -- (215.216bp,89.152bp) -- (253.379bp,111.422bp) -- (291.541bp,97.081bp) -- (329.704bp,96.330bp) -- (367.867bp,89.748bp) -- (406.030bp,79.123bp) -- (444.193bp,58.657bp) -- (482.355bp,44.325bp);
    \markertriangle{kktColor}{62.565}{92.679};
    \markertriangle{kktColor}{100.727}{82.685};
    \markertriangle{kktColor}{138.890}{86.012};
    \markertriangle{kktColor}{177.053}{89.738};
    \markertriangle{kktColor}{215.216}{89.152};
    \markertriangle{kktColor}{253.379}{111.422};
    \markertriangle{kktColor}{291.541}{97.081};
    \markertriangle{kktColor}{329.704}{96.330};
    \markertriangle{kktColor}{367.867}{89.748};
    \markertriangle{kktColor}{406.030}{79.123};
    \markertriangle{kktColor}{444.193}{58.657};
    \markertriangle{kktColor}{482.355}{44.325};
  \end{scope}

  \draw[axisLine] (41.575bp,36.591bp) rectangle (503.345bp,206.740bp);
  \draw[axisLine] (62.565bp,36.591bp) -- (62.565bp,33.101bp);
  \node[tickText,anchor=base west] at (60.018bp,23.513bp) {0};
  \draw[axisLine] (138.890bp,36.591bp) -- (138.890bp,33.101bp);
  \node[tickText,anchor=base west] at (136.343bp,23.513bp) {2};
  \draw[axisLine] (215.216bp,36.591bp) -- (215.216bp,33.101bp);
  \node[tickText,anchor=base west] at (212.669bp,23.513bp) {4};
  \draw[axisLine] (291.541bp,36.591bp) -- (291.541bp,33.101bp);
  \node[tickText,anchor=base west] at (288.995bp,23.513bp) {6};
  \draw[axisLine] (367.867bp,36.591bp) -- (367.867bp,33.101bp);
  \node[tickText,anchor=base west] at (365.320bp,23.513bp) {8};
  \draw[axisLine] (444.193bp,36.591bp) -- (444.193bp,33.101bp);
  \node[tickText,anchor=base west] at (439.099bp,23.513bp) {10};
  \draw[axisLine] (41.575bp,42.723bp) -- (38.085bp,42.723bp);
  \node[tickText,anchor=base west] at (21.841bp,39.683bp) {0.0};
  \draw[axisLine] (41.575bp,69.836bp) -- (38.085bp,69.836bp);
  \node[tickText,anchor=base west] at (21.841bp,66.797bp) {0.5};
  \draw[axisLine] (41.575bp,96.950bp) -- (38.085bp,96.950bp);
  \node[tickText,anchor=base west] at (21.841bp,93.911bp) {1.0};
  \draw[axisLine] (41.575bp,124.064bp) -- (38.085bp,124.064bp);
  \node[tickText,anchor=base west] at (21.841bp,121.025bp) {1.5};
  \draw[axisLine] (41.575bp,151.178bp) -- (38.085bp,151.178bp);
  \node[tickText,anchor=base west] at (21.841bp,148.139bp) {2.0};
  \draw[axisLine] (41.575bp,178.292bp) -- (38.085bp,178.292bp);
  \node[tickText,anchor=base west] at (21.841bp,175.253bp) {2.5};
  \draw[axisLine] (41.575bp,205.406bp) -- (38.085bp,205.406bp);
  \node[tickText,anchor=base west] at (21.841bp,202.367bp) {3.0};
  \node[axisLabel,anchor=base west] at (222.694bp,9.482bp) {Settlement period};
  \node[axisLabel,rotate=90,anchor=base west] at (15.559bp,74.462bp) {Congestion (MW)};

  \draw[legendBox] (378.399bp,170.342bp) rectangle (499.849bp,203.232bp);
  \draw[seriesLine,baseColor] (381.189bp,197.577bp) -- (395.189bp,197.577bp);
  \markercircle{baseColor}{388.189}{197.577};
  \node[legendText,anchor=base west] at (400.789bp,195.127bp) {Base};
  \draw[seriesLine,centColor] (381.189bp,187.311bp) -- (395.189bp,187.311bp);
  \markersquare{centColor}{388.189}{187.311};
  \node[legendText,anchor=base west] at (400.789bp,184.861bp) {Cent.};
  \draw[seriesLine,quadFYColor] (381.189bp,177.045bp) -- (395.189bp,177.045bp);
  \markerdiamond{quadFYColor}{388.189}{177.045};
  \node[legendText,anchor=base west] at (400.789bp,174.595bp) {Quad-FY};
  \draw[seriesLine,kktColor] (444.320bp,187.311bp) -- (458.320bp,187.311bp);
  \markertriangle{kktColor}{451.320}{187.311};
  \node[legendText,anchor=base west] at (463.920bp,184.861bp) {KKT/MILP};
\end{tikzpicture}%
}

\makeatletter

\newlength{\CongNativeWidth}
\newlength{\CongPanelTargetWidth}
\newlength{\CongSquareSlot}

\newcommand{\CongIEEEFinalTickPt}{6.0}
\newcommand{\CongIEEEFinalAxisPt}{7.0}
\newcommand{\CongIEEEFinalLegendPt}{3.6}
\newcommand{\CongIEEEFinalSeriesBp}{0.85}
\newcommand{\CongIEEEFinalAxisBp}{0.45}
\newcommand{\CongIEEEFinalGridBp}{0.25}
\newcommand{\CongIEEEFinalMarkerBp}{1.35}
\newcommand{\CongIEEEFinalMarkerLineBp}{0.50}

\providecommand{\CongCap}[1]{#1}

\newcommand{\CongSetIEEEPanelStyle}[1]{%
  \setlength{\CongPanelTargetWidth}{#1}%
  \pgfmathsetmacro{\CongPanelScale}{\strip@pt\CongPanelTargetWidth/\strip@pt\CongNativeWidth}%
  \pgfmathsetmacro{\CongRawTick}{\CongIEEEFinalTickPt/\CongPanelScale}%
  \pgfmathsetmacro{\CongRawTickLead}{1.18*\CongRawTick}%
  \pgfmathsetmacro{\CongRawAxis}{\CongIEEEFinalAxisPt/\CongPanelScale}%
  \pgfmathsetmacro{\CongRawAxisLead}{1.18*\CongRawAxis}%
  \pgfmathsetmacro{\CongRawLegend}{\CongIEEEFinalLegendPt/\CongPanelScale}%
  \pgfmathsetmacro{\CongRawLegendLead}{1.18*\CongRawLegend}%
  \pgfmathsetmacro{\CongRawSeries}{\CongIEEEFinalSeriesBp/\CongPanelScale}%
  \pgfmathsetmacro{\CongRawAxisLine}{\CongIEEEFinalAxisBp/\CongPanelScale}%
  \pgfmathsetmacro{\CongRawGrid}{\CongIEEEFinalGridBp/\CongPanelScale}%
  \pgfmathsetmacro{\CongRawMarker}{\CongIEEEFinalMarkerBp/\CongPanelScale}%
  \pgfmathsetmacro{\CongRawMarkerDiag}{1.41421356237*\CongRawMarker}%
  \pgfmathsetmacro{\CongRawMarkerLine}{\CongIEEEFinalMarkerLineBp/\CongPanelScale}%
  \tikzset{%
    gridLine/.style={draw=gridGray,line width=\CongRawGrid bp,dash pattern=on 0.5bp off 0.825bp},%
    axisLine/.style={draw=black,line width=\CongRawAxisLine bp},%
    seriesLine/.style={line width=\CongRawSeries bp,line cap=rect,line join=round},%
    legendBox/.style={draw=legendBorder,fill=white,line width=\CongRawAxisLine bp,rounded corners=1.39bp},%
    tickText/.style={font=\normalfont\fontsize{\CongRawTick pt}{\CongRawTickLead pt}\selectfont,inner sep=0bp,outer sep=0bp,text=black},%
    axisLabel/.style={font=\normalfont\fontsize{\CongRawAxis pt}{\CongRawAxisLead pt}\selectfont,inner sep=0bp,outer sep=0bp,text=black},%
    legendText/.style={font=\normalfont\fontsize{\CongRawLegend pt}{\CongRawLegendLead pt}\selectfont,inner sep=0bp,outer sep=0bp,text=black},%
  }%
  \renewcommand{\markercircle}[3]{\path[draw=##1,fill=##1,line width=\CongRawMarkerLine bp] (##2bp,##3bp) circle[radius=\CongRawMarker bp]}%
  \renewcommand{\markersquare}[3]{\path[draw=##1,fill=##1,line width=\CongRawMarkerLine bp] ($ (##2bp,##3bp)+(-\CongRawMarker bp,-\CongRawMarker bp) $) rectangle ($ (##2bp,##3bp)+(\CongRawMarker bp,\CongRawMarker bp) $)}%
  \renewcommand{\markertriangle}[3]{\path[draw=##1,fill=##1,line width=\CongRawMarkerLine bp] (##2bp,{##3bp+\CongRawMarker bp}) -- ({##2bp-\CongRawMarker bp},{##3bp-\CongRawMarker bp}) -- ({##2bp+\CongRawMarker bp},{##3bp-\CongRawMarker bp}) -- cycle}%
  \renewcommand{\markerdiamond}[3]{\path[draw=##1,fill=##1,line width=\CongRawMarkerLine bp] (##2bp,{##3bp-\CongRawMarkerDiag bp}) -- ({##2bp+\CongRawMarkerDiag bp},##3bp) -- (##2bp,{##3bp+\CongRawMarkerDiag bp}) -- ({##2bp-\CongRawMarkerDiag bp},##3bp) -- cycle}%
  \renewcommand{\markercross}[3]{\draw[##1,line width=\CongRawMarkerLine bp,line cap=round] ({##2bp-\CongRawMarker bp},{##3bp-\CongRawMarker bp}) -- ({##2bp+\CongRawMarker bp},{##3bp+\CongRawMarker bp}) ({##2bp-\CongRawMarker bp},{##3bp+\CongRawMarker bp}) -- ({##2bp+\CongRawMarker bp},{##3bp-\CongRawMarker bp})}%
}

\newcommand{\IEEECongPanel}[2][\columnwidth]{%
  \begingroup
    \CongSetIEEEPanelStyle{#1}%
    \resizebox{#1}{!}{#2}%
  \endgroup
}

\newcommand{\IEEECongA}[1][\columnwidth]{\IEEECongPanel[#1]{\FigureCaseA}}
\newcommand{\IEEECongB}[1][\columnwidth]{\IEEECongPanel[#1]{\FigureCaseB}}
\newcommand{\IEEECongC}[1][\columnwidth]{\IEEECongPanel[#1]{\FigureCaseC}}
\newcommand{\IEEECongRone}[1][\columnwidth]{\IEEECongPanel[#1]{\FigureCaseRone}}

\makeatother

\newtheorem{assumption}{Assumption}

\newtheorem{theorem}{Theorem}

\newtheorem{proposition}{Proposition}

\newtheorem{remark}{Remark}

\newcommand{\R}{\mathbb{R}}
\newcommand{\Spsd}{\mathbb{S}_{+}}
\newcommand{\Spd}{\mathbb{S}_{++}}
\newcommand{\Hubs}{\mathcal{H}}
\newcommand{\Time}{\mathcal{T}}
\newcommand{\Lines}{\mathcal{L}}
\newcommand{\Buses}{\mathcal{N}}
\newcommand{\X}{\mathcal{X}}
\newcommand{\D}{\mathcal{D}}

\newcommand{\dd}{\Delta t}
\newcommand{\pos}[1]{\left[#1\right]_{+}}

\newcommand{\dom}{\operatorname{dom}}

\newcommand{\col}{\operatorname{col}}

\newcommand{\argmin}{\operatorname*{arg\,min}}

\newcommand{\ones}{\mathbf{1}}

\newcolumntype{L}[1]{>{\raggedright\arraybackslash}p{#1}}
\newif\ifshownomenclature
\shownomenclaturetrue 

\usepackage[hidelinks]{hyperref}
\usepackage{orcidlink}

\newcommand{\equalcontrib}{\textsuperscript{\ensuremath{\dagger}}}
\newcommand{\authororcid}[1]{\,\orcidlink{#1}}

\usepackage{enumitem}

\begin{document}


\title{Dynamic Congestion Pricing in Distribution Networks via a Convex-Analytic Bilevel Reformulation}
\author{Reza~Rahimi~Baghbadorani\authororcid{0000-0001-6058-2485}\equalcontrib,
        Ali~Nikseresht\authororcid{0000-0002-6107-7699}\equalcontrib,~\IEEEmembership{Member,~IEEE,}
        Jehum~Cho\authororcid{0000-0003-3962-9095},
        and~Yashar~Ghiassi-Farrokhfal\authororcid{0000-0001-6365-1001}%
\thanks{\equalcontrib Reza Rahimi Baghbadorani and Ali Nikseresht contributed equally to this work.}%
\thanks{Reza Rahimi Baghbadorani, Ali Nikseresht, and Yashar Ghiassi-Farrokhfal are with the Rotterdam School of Management, Erasmus University, 3062 PA, Rotterdam, The Netherlands (e-mail: rahimibaghbadorani@rsm.nl; ali.nikseresht@ieee.org; y.ghiassi@rsm.nl).}%
\thanks{Jehum Cho is with the Erasmus School of Economics, Erasmus University, 3062 PA, Rotterdam, The Netherlands (e-mail: j.cho@ese.eur.nl).}%
\thanks{This work was supported by NGF--GroenvermogenNL 2022 under Grant NGF.1611.22.021.}%
}


\markboth{}{Rahimi Baghbadorani \MakeLowercase{\textit{et al.}}: Fenchel--Young Sequential Convex Approximation for Bilevel Dynamic Congestion Pricing}

\maketitle

\begin{abstract}
Dynamic congestion pricing is an important tool for managing congestion and coordinating distributed energy resources in active distribution networks. However, scalable mechanisms that preserve participant autonomy remain computationally challenging because the operator--resource interaction is naturally bilevel. This paper develops a convex-analytic framework in which a distribution system operator computes dynamic congestion-price adders, while decentralized energy hubs schedule flexible demand, storage, local generation, renewable curtailment, and grid import/export. Unlike conventional single-level reformulations that replace lower-level problems by Karush--Kuhn--Tucker (KKT) conditions, complementarity constraints, and big-\(M\) linearizations, the proposed model represents follower feasibility and optimality through a Fenchel--Young equality involving the convex conjugate of an extended follower objective. The remaining bilinear price-response term is handled through a penalized difference-of-convex reformulation and sequential convex approximation. The method solves continuous convex subproblems and avoids the constraint-wise complementarity and branch-and-bound scaling of mixed-integer KKT reformulations; its main computational drivers are price-response dimension and conjugate evaluation rather than binary encodings of follower inequalities. On augmented IEEE 13- and 34-node feeders, it reduces congestion by \(96.89\%\) and \(96.45\%\), respectively, approaches centralized full-information dispatch, certifies price-response consistency to numerical precision, and yields lower residual congestion than time-limited KKT incumbents within the computational budget.
\end{abstract}


\ifshownomenclature
\begin{table*}[htb]
\centering
{\bfseries Nomenclature\par}
\vspace{0.6ex}
\setlength{\tabcolsep}{2.3pt}
\renewcommand{\arraystretch}{0.92}
{\scriptsize
\begin{tabular*}{\textwidth}{@{\extracolsep{\fill}}L{0.155\textwidth}L{0.325\textwidth}L{0.155\textwidth}L{0.325\textwidth}@{}}
\toprule
\textbf{Symbol} & \textbf{Meaning} & \textbf{Symbol} & \textbf{Meaning}\\
\midrule
\multicolumn{2}{@{}l}{\emph{Sets, prices, and network quantities}} & \multicolumn{2}{l@{}}{\emph{Hub variables, convex analysis, and algorithms}}\\
$\Hubs,\Time,\Lines,\Buses$ & Hubs, time periods, monitored lines, and buses. & $x_i,z_i,\X_i$ & Implemented schedule, candidate schedule, and feasible set of hub $i$.\\
$i,t,\ell;T,\dd$ & Hub, time, and line indices; horizon length and market step. & $A_i x_i\le b_i$, $E_i x_i=d_i$ & Local inequality/equality constraints of hub $i$.\\
$\R^n,\mathbb S_+^n,\mathbb S_{++}^n$ & Euclidean space, positive semidefinite cone, and positive definite cone. & $m_{it}^{\rm in},m_{it}^{\rm out},w_{it}$ & Import, export, and net withdrawal, $w_{it}=m_{it}^{\rm in}-m_{it}^{\rm out}$.\\
$\delta_{it},\delta_i,\D$ & DSO congestion-price adder, hub price-adder vector, and admissible set. & $D_{it}^{\rm fix},u_{it},u_{it}^{0},a_{it}$ & Fixed load, flexible load, baseline flexible load, and absolute deviation.\\
$\underline\delta_i,\bar\delta_i$ & Price lower/upper bounds. & $p_{it}^{\rm ch},p_{it}^{\rm dis},e_{it}$ & Battery charge, discharge, and state of energy.\\
$\lambda_t^{\rm buy},\lambda_t^{\rm sell}$ & Reference buy/sell prices. & $PV_{it}^{\rm avail},c_{it}^{\rm pv},g_{it}$ & PV availability, PV curtailment, and dispatchable generation.\\
$\tau_i^{\rm base},\pi_{it}^{\rm buy},\pi_{it}^{\rm sell}$ & Base tariff and effective buy/sell prices. & $\chi^{\rm sell}$ & Export-credit pass-through of the congestion adder.\\
$\lambda_i(\delta_i)$ & Affine price coefficient in the follower objective. & $\eta_i^{\rm ch},\eta_i^{\rm dis}$ & Battery efficiencies.\\
$C^{\rm line},C^{\rm sub},C^{\rm total}$ & Reported line, substation, and total overload residuals. & $E^{\rm flex,shift},E^{\rm bat,thr}$ & Flexible-load shift and battery-throughput metrics.\\
$f_{\ell t},f_{\ell t}^{\rm bg},H_{\ell i}$ & Total flow, background flow, and PTDF coefficient. & $M_i^{\rm in},M_i^{\rm out}$ & Selection matrices for import/export components.\\ $S_t,S_t^{\rm bg},\bar f_\ell,\bar S^{\pm}$ & Substation exchange, background exchange, line limit, and substation import/export limits. & $I_{\X_i},\psi_i,\phi_i$ & Indicator, base follower cost, and extended objective $\phi_i=\psi_i+I_{\X_i}$.\\
$o_{\ell t}^{\rm line},o_t^{\rm sub}$ & Line and substation overload epigraph variables. & $\phi_i^*,\phi_i^{**},\partial\phi_i$ & Convex conjugate, biconjugate, and subdifferential.\\
$\Gamma,J^{\rm cong},J^{\rm op},R$ & Overload penalty, congestion metric, operating-cost proxy, and price regularizer. & $D_{\phi_i}$ & Fenchel--Young gap.\\
$\mu_i,\nu_i$ & KKT multipliers used only in the MPEC benchmark. & $Q_i,q_i$ & Quadratic follower-cost matrix and linear coefficient.\\
$[\cdot]_+$, $\col(\cdot)$ & Positive part and vector-stacking operator. & $h_i,o_i,\widehat D_i^k$ & DC components and majorized gap at iteration $k$.\\ $\rho_k,\rho_{\max};\ \gamma_\rho,\eta_\rho;\ K_{\max}$ & Penalty, cap, growth factor, stall threshold, and iteration limit.  & $r_i,r_{\max},s^k$; $\varepsilon_{\rm gap},\varepsilon_{\rm step}$ & Hub residual, maximum residual, step diagnostic, and tolerances.\\ $\alpha^{\rm line},\alpha^{\rm sub}$ & Line- and substation-congestion penalty weights. & $\beta^{\rm level},\beta^{\rm fair}$ & Price-level and cross-hub fairness weights.\\
\bottomrule
\end{tabular*}}
\vspace{-1.2ex}
\end{table*}
\fi


\section{Introduction}
\label{sec:introduction}

\IEEEPARstart{D}{istribution} networks are undergoing a structural transition from passively operated radial feeders to actively managed cyber-physical systems in which distributed energy resources (DERs), electric vehicles, batteries, heat pumps, flexible demand, and multi-energy hubs materially affect network loading and local market outcomes \cite{sotkiewicz2006nodal,sundstrom2012ev,nick2014storage,mancarella2013realtime,ferro2026embedded}. This transition has moved congestion management from a predominantly planning-oriented concern to a recurring operational problem for distribution system operators (DSOs), because privately operated devices can create overloads during peak consumption, concentrated electric-vehicle charging, or high distributed-generation export periods \cite{liu2014dcp,huang2016dynamic}. Demand response (DR) and price-responsive flexibility are therefore central instruments for improving network utilization, integrating renewable resources, deferring reinforcement, and reducing operating costs \cite{deng2015survey,pandey2022hierarchical}.

Price-based coordination is attractive because it can influence flexible resources while preserving a degree of autonomy for customers, aggregators, and energy hubs \cite{schweppe1988spot,corradi2013price}. Distribution nodal-pricing, distribution congestion price (DCP), and dynamic-tariff mechanisms seek to communicate network scarcity through economically interpretable signals rather than through direct control \cite{sotkiewicz2006nodal,liu2014dcp,huang2016dynamic}. Dynamic tariffs have been developed for household DR, electric-vehicle charging, feeder reconfiguration, line-loss reduction, and load-serving-entity pricing in distribution systems \cite{sundstrom2012ev,liu2014dcp,huang2016dynamic,nguyen2016dynamic,pandey2022hierarchical}. Related retail and incentive mechanisms, including stepwise tariffs, coupon-based DR, real-time pricing, Stackelberg supply-demand balancing, hierarchical incentive-based DR, reinforcement-learning-based dynamic pricing, and double-signal pricing for batteries, show that price design must balance network efficiency, customer response, implementability, and economic acceptability \cite{li2013stepwise,zhong2013coupon,corradi2013price}.

The interaction between a price-setting operator and price-responsive participants is naturally hierarchical. Consequently, Stackelberg games, bilevel programming, and multi-level optimization have become the dominant modeling frameworks for dynamic pricing and incentive design in power systems. These frameworks have been widely applied to electricity-market pricing, distribution network management, demand-side flexibility, electric-vehicle charging, virtual power plant (VPP) coordination, multi-microgrid operation, and DER aggregation \cite{dempe2002foundations,dempe2013bilevel,tushar2012ev,wei2015pricing,asimakopoulou2013leader,asimakopoulou2015hierarchical,zhang2018procurement,vahid2019self,yi2020vpp_scheduling,yi2020adn,parizy2020threelevel,yi2021vpp_der,xu2022competitive}. In these formulations, the system operator determines prices, incentives, or dispatch decisions, while decentralized participants optimize their own objectives subject to local operational constraints. Decomposition, distributed optimization, and column-generation techniques have further improved the scalability of these models by exploiting separability among heterogeneous resources. However, while such approaches can accelerate computation, they do not eliminate the equilibrium constraints that arise when prices are decision variables, leaving the fundamental computational challenge of bilevel pricing largely intact \cite{dantzig1960decomposition,anjos2019decentralized}.


Despite their widespread adoption, bilevel pricing problems remain computationally challenging. The dominant solution strategy replaces the lower-level optimization problem with its Karush--Kuhn--Tucker (KKT) conditions, strong duality, or equivalent primal--dual optimality conditions, producing a mathematical program with equilibrium constraints (MPEC) or, after complementarity linearization, a mixed-integer reformulation \cite{wei2015pricing,nguyen2016dynamic,zhang2018procurement,yi2020adn,yi2021vpp_der}. Although this approach is exact under suitable convexity and regularity assumptions, it introduces dual variables and complementarity constraints whose size grows with the complexity of the lower-level operational model rather than with the pricing decisions themselves \cite{dempe2002foundations,dempe2013bilevel}. Moreover, multi-period tariff optimization is intrinsically difficult—even simplified formulations are NP-hard—motivating approximation, decomposition, and structure-exploiting reformulations for practical applications \cite{kovacs2018complexity}. These computational challenges become particularly pronounced when lower-level models capture detailed DERs with intertemporal operating constraints, including flexible demand, storage dynamics, renewable curtailment, import/export decisions, dispatchable generation, and other flexible energy resources \cite{mancarella2013realtime,li2021cies}.

This paper addresses this challenge by developing a convex-analytic reformulation for bilevel dynamic congestion pricing in active distribution networks with affine price-responsive followers. The central idea is to encode each decentralized hub's feasibility and optimality through a Fenchel--Young (FY) optimality gap associated with an extended convex follower objective, rather than expanding the lower-level problem into KKT stationarity and complementarity conditions. The exact optimality-gap representation remains nonconvex in the joint price-response variables; therefore, the proposed method treats the residual price-response coupling through a difference-of-convex (DC) decomposition and sequential convex approximation (SCA). This yields continuous convex subproblems, computable lower-level consistency residuals, and a scalable alternative to complementarity-based KKT/MPEC reformulations. 
Accordingly, we consider an active distribution network in which the DSO determines local congestion price adders while decentralized energy hubs optimize flexible demand, storage, local generation, renewable curtailment, and import/export. The resulting net withdrawals are mapped to feeder line and substation loading through a PTDF-based network model. Fig.~\ref{fig:conceptual_architecture} illustrates the resulting leader–follower architecture and the interaction between network pricing, decentralized optimization, and the proposed FY optimality-gap representation.

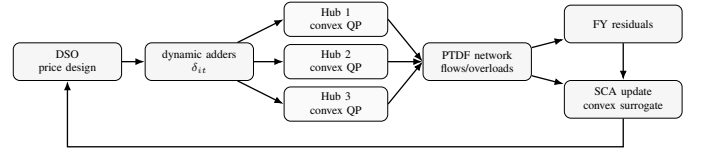
\begin{figure}[htb]
\centering
\resizebox{\columnwidth}{!}{%
\begin{tikzpicture}[
    font=\scriptsize,
    box/.style={draw, rounded corners, align=center, minimum width=23mm, minimum height=8mm, fill=gray!7},
    hub/.style={draw, rounded corners, align=center, minimum width=22mm, minimum height=7mm, fill=gray!4},
    cert/.style={draw, rounded corners, align=center, minimum width=25mm, minimum height=8mm, fill=gray!7},
    arr/.style={-{Latex[length=2mm]}, thick}
]
\node[box] (dso) at (0,0) {DSO\\price design};
\node[box] (price) at (2.8,0) {dynamic adders\\$\delta_{it}$};
\node[hub] (h1) at (5.7,0.9) {Hub 1\\convex QP};
\node[hub] (h2) at (5.7,0) {Hub 2\\convex QP};
\node[hub] (h3) at (5.7,-0.9) {Hub 3\\convex QP};
\node[box] (net) at (8.7,0) {PTDF network\\flows/overloads};
\node[cert] (loss) at (11.8,0.8) {FY residuals};
\node[cert] (update) at (11.8,-0.8) {SCA update\\convex surrogate};
\draw[arr] (dso) -- (price);
\draw[arr] (price) -- (h1.west);
\draw[arr] (price) -- (h2.west);
\draw[arr] (price) -- (h3.west);
\draw[arr] (h1.east) -- (net.west);
\draw[arr] (h2.east) -- (net.west);
\draw[arr] (h3.east) -- (net.west);
\draw[arr] (net) -- (loss);
\draw[arr] (net) -- (update);
\draw[arr] (loss) -- (update);
\draw[arr] (update.south) |- ++(0,-0.6) -| (dso.south);
\end{tikzpicture}}
\caption{A DSO--hub architecture example. Price adders are leader decisions; hubs solve convex price-response problems; induced withdrawals determine network congestion; and FY residuals certify lower-level consistency.}
\label{fig:conceptual_architecture}
\end{figure}

The main contributions are as follows.
\begin{itemize}
    \item We develop a bilevel dynamic congestion-pricing framework for active distribution networks in which a distribution system operator coordinates decentralized, price-responsive energy hubs through time-varying congestion price adders. The formulation captures realistic hub operation, including flexible demand, battery storage, renewable curtailment, dispatchable generation, and network-constrained power exchanges, while consistently linking hub decisions to distribution-network congestion.
    \item We derive a convex-analytic reformulation of lower-level optimality for bilevel pricing problems with affine price entry using the FY optimality gap. The proposed representation embeds feasibility through convex conjugacy and avoids explicit KKT multipliers, complementarity constraints, and big-\(M\) linearizations that are characteristic of conventional MPEC reformulations.
    \item We develop an FY-SCA solution framework by combining the FY reformulation with a difference-of-convex decomposition and sequential convex approximation. The resulting algorithm solves a sequence of convex optimization problems while providing residual-based certificates of lower-level optimality.
    \item We validate the proposed framework on IEEE 13- and 34-node active distribution feeders and compare it against baseline operation, centralized dispatch, and KKT/MPEC bilevel pricing. The numerical study demonstrates substantial congestion relief, near-centralized performance, low optimality-gap residuals, and significantly improved computational efficiency relative to conventional KKT-based reformulations.
\end{itemize}

The remainder of this paper is organized as follows. Section~\ref{sec:preliminaries} states the bilevel pricing model and convex-analytic notation. Section~\ref{sec:fy_reformulation} derives the FY optimality-loss reformulations and the SCA algorithm. Section~\ref{sec:congestion_specialization} tailors the method to distribution-network congestion pricing. Section~\ref{sec:numerical_results} presents the numerical protocol and results, and Section~\ref{sec:conclusion} concludes the paper. The notation used throughout the paper is summarized in the Nomenclature.





\section{Convex-Analytic Bilevel Pricing Model}
\label{sec:preliminaries}
This section formalizes the convex-analytic modeling framework underlying the proposed bilevel pricing approach. We consider a class of bilevel pricing problems in which a leader determines a price vector and decentralized participants respond by solving convex optimization problems. The key idea is to absorb the follower feasibility constraints into an extended-valued objective, allowing lower-level optimality to be represented through a single FY equality rather than an expanded system of KKT conditions. We first introduce the convex-analytic notation required for this construction. Afterwards, we use the FY inequality to elaborate on price-response optimality as a nonnegative gap. We then express the generic bilevel structure. Finally, we embed this representation into a generic DSO's dynamic congestion-pricing model.

\subsection{Convex-Analytic Preliminaries}
\label{subsec:convex_preliminaries}

Let \(\R^n\) denote the \(n\)-dimensional Euclidean space. The inner product is \(\langle x,y\rangle=x^\top y\), and \(\|x\|\) denotes the Euclidean norm unless otherwise stated.

A function \(\phi:\R^n\to\R\cup\{+\infty\}\) is proper if its effective domain \(\dom\phi=\{x:\phi(x)<+\infty\}\) is nonempty and \(\phi(x)>-\infty\) for all \(x\). For a nonempty closed convex set \(\X\), its indicator function is
\begin{equation}
I_{\X}(x)=
\begin{cases}
0, & x\in\X,\\
+\infty, & x\notin\X.
\end{cases}
\label{eq:indicator_function}
\end{equation}
The convex conjugate of a proper function \(\phi\) is
\begin{equation}
    \phi^*(y)=\sup_{x\in\R^n}\{y^\top x-\phi(x)\}.
    \label{eq:conjugate_definition}
\end{equation}
If \(\phi\) is proper, closed, and convex, then \(\phi^{**}=\phi\) under the lower-semicontinuity assumptions \cite{rockafellar1970convex}. The subdifferential of \(\phi\) at \(x\in\dom\phi\) is
\begin{equation}
    \partial\phi(x)=\{y\in\R^n:\phi(z)\ge \phi(x)+y^\top(z-x),\;\forall z\in\R^n\}.
    \label{eq:subdifferential}
\end{equation}

\subsection{Fenchel--Young Gap and Optimality Conditions}
\label{subsec:fy_gap}

For any proper closed convex function \(\phi\), the FY inequality states
\begin{equation}
    \phi(x)+\phi^*(y)-y^\top x\ge0,
    \label{eq:fy_ineq}
\end{equation}
with equality if and only if \(y\in\partial\phi(x)\)\cite{rockafellar1970convex,boyd2004convex}. The FY divergence associated with \(\phi\) is
\begin{equation}
    D_\phi(x,y)=\phi(x)+\phi^*(y)-y^\top x.
    \label{eq:fy_divergence}
\end{equation}
Thus, \(D_\phi(x,y)\ge0\), and \(D_\phi(x,y)=0\) when \(y\in\partial\phi(x)\). In a price-response problem of the form \(\min_z\{\phi(z)+\lambda^\top z\}\), the first-order optimality condition is \(-\lambda\in\partial\phi(x)\). Therefore, follower optimality is equivalent to
\begin{equation}
    D_\phi(x,-\lambda)=\phi(x)+\phi^*(-\lambda)+\lambda^\top x=0.
    \label{eq:fy_gap_price}
\end{equation}

\noindent \textit{Notation}:
For vectors \(a\) and \(b\), \(\col(a,b)\) denotes vertical concatenation. The notation \(M\in\mathbb{S}_{+}^{n}\) means that \(M\) is positive semidefinite. The positive part is \(\pos{r}=\max\{r,0\}\). For a convex set \(\D\), \(\Pi_\D(y)\) denotes the Euclidean projection of \(y\) onto \(\D\), whenever the projection is used algorithmically. All vectors are column vectors. Inequalities between vectors are componentwise. Unless stated otherwise, all optimization problems are deterministic over a finite horizon.

\subsection{Affine Price-Entry in Bilevel Structure}
\label{sec:generic_template}
\subsubsection{Generic bilevel problem}
\label{subsec:generic_bilevel_model}

Consider an upper-level price or incentive vector $\lambda$
and a lower-level response \(x(\lambda)\). The affine price-entry bilevel problem is
\begin{equation}
\begin{aligned}
\min_{\lambda\in\Lambda}\quad &F(x(\lambda),\lambda)\\
\text{s.t.}\quad &x(\lambda)\in\argmin_{z\in\X}\{\psi(z)+\lambda^\top z\},
\end{aligned}
\label{eq:generic_bilevel}
\end{equation}
where \(F\) is the leader objective, $\Lambda\subseteq\R^{n}$ and \(\X\subseteq\R^n\) are the leader's and followers' feasible sets, respectively, and we assume they are nonempty, closed, and convex. Strongly convex function \(\psi:\R^n\to\R\cup\{+\infty\}\) represents the price-independent part of the follower cost, disutility, degradation, or negative utility. The equal dimensions of \(x\) and \(\lambda\) in \eqref{eq:generic_bilevel} are used only for clarity. If the leader variable is \(\eta\in\R^p\) and prices enter as \(B\eta+d\), then one replaces \(\lambda\) by the affine coefficient vector \(B\eta+d\) throughout the analysis.

In the computational model and KKT benchmark, we use the polyhedral follower set
\begin{equation}
    \X=\{z\in\R^n:A z\le b,\;E z=d\},
    \label{eq:generic_polyhedron}
\end{equation}
where \(A\in\R^{m_I\times n}\), \(E\in\R^{m_E\times n}\), \(b\in\R^{m_I}\), and \(d\in\R^{m_E}\). This polyhedral structure refers only to the local hub/VPP operating constraints. The FY reformulation remains valid for any other proper, closed, convex follower objective with affine price entry. Defining
\begin{equation}
    \phi(z)=\psi(z)+I_{\X}(z),
    \label{eq:generic_phi_definition}
\end{equation}
the lower-level problem can be written equivalently as
\(x(\lambda)\in\argmin_z\{\phi(z)+\lambda^\top z\}\).

\subsubsection{Quadratic prototype}
\label{subsec:quadratic_prototype}

A key special case is
\begin{equation}
    \psi(z)=\frac12 z^\top Qz+q^\top z,
    \qquad Q\in\mathbb{S}_{+}^{n},
    \label{eq:quadratic_prototype_clean}
\end{equation}
with \(z\in\X\). If \(Q \in \mathbb{S}_{++}^{n}\) on the feasible affine hull of \(\X\), the follower response is unique. However, if \(Q \in \mathbb{S}_{+}^{n}\), the FY reformulation remains valid, but the response may be set-valued and the leader model should be interpreted in the optimistic or selected-response sense unless a tie-breaking regularizer is added.

The quadratic specification is a modeling and regularization choice rather than a requirement of the FY identity. It provides a tractable representation of marginal discomfort, deviation from preferred schedules, battery degradation, and empirically identified price sensitivity. It is also economically meaningful: many demand-response and DER decisions exhibit diminishing marginal benefit or increasing marginal adjustment cost, so larger deviations from preferred operation become progressively less attractive. In distribution-level DLMP design, price sensitivity similarly yields strictly convex quadratic aggregator problems with unique responses \cite{huang2015dlmp_qp}. If \(Q=0\), the follower becomes a linear program and the FY equality remains exact; only uniqueness may be lost. A selected-response convention or an arbitrarily small term \((\epsilon/2)\|z\|_2^2\), \(\epsilon>0\), restores a single-valued response when required.

\begin{remark}[Role of leader convexity]
Joint convexity of \(F(x,\lambda)\) is useful for the convex SCA subproblems, but it does not make the exact single-level FY reformulation convex, because the optimality equality contains the bilinear term \(\lambda^\top x\). Therefore, convexity claims in this paper are made for the majorized subproblems, not for the exact bilevel-equivalent formulation.
\end{remark}

\subsection{Dynamic Congestion-Pricing Model}
\label{sec:model}

In this subsection, we introduce a wider dynamic pricing model perspective, and then later in section \ref{sec:congestion_specialization} we tailor it to a congestion problem, which concerns the numerical studies.

\subsubsection{Leader--follower architecture}
\label{subsec:leader_follower_architecture}

Consider a DSO that coordinates a set of energy hubs \(\Hubs\) over time horizon \(\Time\). Here \(i\in\Hubs\) indexes hubs, \(t\in\Time\) indexes market intervals, and \(\dd\) denotes the interval length. The DSO chooses a dynamic congestion-price vector
\begin{equation}
    \delta=\col(\delta_i:i\in\Hubs),\qquad
    \delta_i=\col(\delta_{it}:t\in\Time),
    \label{eq:delta_definition}
\end{equation}
from an admissible price set \(\D\). Each hub \(i\) observes \(\delta_i\) and solves a local scheduling problem. The induced response is denoted by \(x_i(\delta_i)\), and the aggregate response by \(x(\delta)=\col(x_i(\delta_i):i\in\Hubs)\).

The admissible dynamic-price set is modeled as
\begin{equation}
\begin{aligned}
    \D=\{\delta:\;&\underline\delta_i\le \delta_{it}\le \bar\delta_i,\quad \forall i,t\},
\end{aligned}
\label{eq:price_set}
\end{equation}
where \(\underline\delta_i\) and \(\bar\delta_i\) denote the lower and upper admissible adders for hub \(i\). Additional market-design constraints, such as average-revenue neutrality, price-zone restrictions, or cross-hub fairness caps, can be appended to \(\D\) as convex constraints.

\subsubsection{Follower response model}
\label{subsec:follower_response}

For each hub \(i\), let \(x_i\in\R^{n_i}\) collect all local variables over the horizon. The hub response is represented by the convex quadratic problem
\begin{equation}
\begin{aligned}
    x_i(\delta_i)\in\argmin_{z_i\in\X_i}\quad
    &\frac12 z_i^\top Q_i z_i+q_i^\top z_i+\lambda_i(\delta_i)^\top z_i,
\end{aligned}
\label{eq:follower_qp}
\end{equation}
where \(Q_i\in\Spsd^{n_i}\), \(q_i\in\R^{n_i}\), \(\X_i\) is a nonempty closed convex feasible set, and \(\lambda_i(\delta_i)\in\R^{n_i}\) is the vector of price-induced linear coefficients. In the energy-hub model, \(\lambda_i(\delta_i)\) is constructed from \(\pi^{\rm buy}_{it}\), \(\pi^{\rm sell}_{it}\), and import/export selection matrices.

Define
\begin{equation}
    \phi_i(z_i)=\frac12 z_i^\top Q_i z_i+q_i^\top z_i+I_{\X_i}(z_i).
    \label{eq:phi_i_definition}
\end{equation}
Then \eqref{eq:follower_qp} can be written compactly as
\begin{equation}
    x_i(\delta_i)\in\argmin_{z_i\in\R^{n_i}}\{\phi_i(z_i)+\lambda_i(\delta_i)^\top z_i\}.
    \label{eq:follower_phi_i}
\end{equation}
The function \(\phi_i\) includes all local feasibility constraints through the indicator function. This is the key step that makes the FY representation compact.


\subsubsection{Upper-level congestion-management objective}
\label{subsec:upper_level_objective}

The DSO's objective combines network congestion, operating-cost proxies, and price-regularization terms. A generic bilevel dynamic-pricing problem is
\begin{equation}
\begin{aligned}
\min_{\delta\in\D}\quad &F(x(\delta),\delta)\\
\text{s.t.}\quad &x_i(\delta_i)\in\argmin_{z_i\in\R^{n_i}}\{\phi_i(z_i)+\lambda_i(\delta_i)^\top z_i\},\quad \forall i\in\Hubs .
\end{aligned}
\label{eq:bilevel_pricing}
\end{equation}
The function \(F\) is specified in Section~\ref{sec:congestion_specialization}. For the theoretical development, it is sufficient that \(F\) be proper, closed, and convex in the variables of each SCA subproblem. This covers squared or Huber overload penalties, convex operating-cost terms, and quadratic price regularization.

\begin{assumption}[Standing assumptions]\label{subsec:standing_assumptions}
We impose the following assumptions throughout the paper:
   \begin{enumerate}[label=(\roman*), itemsep = 0mm, topsep = 0mm, leftmargin = 7mm]
    \item \label{ass:convex_price_response}
    {\bf{Convex price-response structure}:} For each \(i\in\Hubs\), \(\X_i\subseteq\R^{n_i}\) is nonempty, closed, and convex; \(Q_i\in\Spsd^{n_i}\); \(\phi_i\) in \eqref{eq:phi_i_definition} is proper, closed, and convex; and the price enters the follower objective affinely through \(\lambda_i(\delta_i)^\top z_i\).

    \item \label{ass:wellposed} {\bf{Well-posed follower response}:} For every \(\delta\in\D\), the follower problem \eqref{eq:follower_qp} has at least one optimal solution. If uniqueness is needed, either \(Q_i\in\Spd^{n_i}\) on the feasible affine hull of \(\X_i\), or a small positive definite tie-breaking regularizer is added to \(Q_i\). With \(Q_i=0\), the formulation remains valid but the response may be nonunique unless a selection rule is specified.

    \item \label{ass:leader_convex}{\bf{Convex leader feasible set and objective}:} The price set \(\D\) is nonempty, closed, and convex. The leader's objective used in each convex approximation is proper, closed, and convex in its explicit decision variables.

    \item \label{ass:conjugate_regular}{\bf{Conjugate regularity}:} For each \(i\), the conjugate \(\phi_i^*\) is finite on the relevant price-coefficient domain, or its epigraph admits a closed convex representation over that domain. Standard relative-interior constraint qualifications are assumed whenever an explicit dual representation of \(\phi_i^*\) is used.
    \end{enumerate}
\end{assumption}





\section{Optimality-Loss Reformulation and Sequential Convex Approximation}
\label{sec:fy_reformulation}

\subsection{Embedding Feasibility Into the Follower Objective}
\label{subsec:embedded_feasibility}

The indicator representation in \eqref{eq:phi_i_definition} moves all local feasibility constraints into \(\phi_i\). Thus, the follower problem is unconstrained in the extended-real-valued sense. This step permits primal feasibility, convex costs, and intertemporal device constraints to be treated through one closed convex object. For a polyhedral hub model,
\begin{equation}
   \X_i=\{z_i:A_i z_i\le b_i,\;E_i z_i=d_i\},
    \label{eq:polyhedral_Xi}
\end{equation}
where \(A_i\) and \(E_i\) collect inequality and equality constraints, respectively. The indicator \(I_{\X_i}\) then includes power-balance equations, energy-conservation constraints, storage dynamics, and operational bounds.

\subsection{Exact Optimality-Gap Characterization}
\label{subsec:exact_gap_characterization}

\begin{theorem}[Exact Fenchel--Young reformulation]
\label{thm:fy_exact}
Suppose Assumptions~\ref{subsec:standing_assumptions}--\ref{ass:convex_price_response} and \ref{subsec:standing_assumptions}--\ref{ass:wellposed} hold. For any \(i\in\Hubs\) and any pair \((x_i,\delta_i)\), the following statements are equivalent:
\begin{enumerate}
    \item \(x_i\in\argmin_{z_i}\{\phi_i(z_i)+\lambda_i(\delta_i)^\top z_i\}\).
    \item \(-\lambda_i(\delta_i)\in\partial\phi_i(x_i)\).
    \item \(D_{\phi_i}(x_i,-\lambda_i(\delta_i))=0\), i.e.,
    \begin{equation}
    \phi_i(x_i)+\phi_i^*(-\lambda_i(\delta_i))+\lambda_i(\delta_i)^\top x_i=0.
    \label{eq:fy_equal_i}
    \end{equation}
\end{enumerate}
Consequently, \eqref{eq:bilevel_pricing} is equivalent to the single-level problem
\begin{equation}
\begin{aligned}
\min_{x,\delta}\quad &F(x,\delta)\\
\text{s.t.}\quad &\delta\in\D,\\
&D_{\phi_i}(x_i,-\lambda_i(\delta_i))=0,
\quad \forall i\in\Hubs .
\end{aligned}
\label{eq:single_level_fy_exact}
\end{equation}
\end{theorem}

\begin{proof}
The first statement is equivalent to
\begin{equation}
    \phi_i(x_i)+\lambda_i(\delta_i)^\top x_i
    \le
    \phi_i(z_i)+\lambda_i(\delta_i)^\top z_i,
    \quad \forall z_i.
\end{equation}
Equivalently,
\begin{equation}
    \phi_i(x_i)+\lambda_i(\delta_i)^\top x_i
    -\inf_{z_i}\{\phi_i(z_i)+\lambda_i(\delta_i)^\top z_i\}\le0.
\end{equation}
Using the definition of the conjugate,
\begin{equation}
    -\inf_{z_i}\{\phi_i(z_i)+\lambda_i(\delta_i)^\top z_i\}
    =\phi_i^*(-\lambda_i(\delta_i)).
\end{equation}
Thus, follower optimality implies
\begin{equation}
    \phi_i(x_i)+\phi_i^*(-\lambda_i(\delta_i))+\lambda_i(\delta_i)^\top x_i\le0.
\end{equation}
FY inequality gives the reverse inequality for all \((x_i,\delta_i)\). Hence equality holds. Equality in FY inequality is equivalent to \(-\lambda_i(\delta_i)\in\partial\phi_i(x_i)\), proving equivalence. Substituting this condition for each follower in \eqref{eq:bilevel_pricing} gives \eqref{eq:single_level_fy_exact}.
\end{proof}

\begin{remark}[Exactness and nonconvexity]
The formulation \eqref{eq:single_level_fy_exact} is exact under the stated convexity and existence assumptions. It is not generally convex because \(\lambda_i(\delta_i)^\top x_i\) is bilinear when both the price and the response are decision variables. Therefore, the role of the next section is not to claim that the exact single-level problem is convex, but to construct a disciplined convex decomposition of a penalized version of it.
\end{remark}

\noindent\textbf{Conjugate evaluation for quadratic-polyhedral followers:} For the quadratic hub model, the conjugate in \eqref{eq:fy_equal_i} can be evaluated or represented through a continuous convex dual problem. If
\(\X_i=\{z_i:A_i z_i\le b_i,\ E_i z_i=d_i\}\) and \(Q_i\in\Spd^{n_i}\), then for any coefficient \(y_i\),
\begin{equation}
\begin{aligned}
\phi_i^*(y_i)=\min_{\mu_i\ge0,\nu_i}\;&\frac12 v_i^\top Q_i^{-1}v_i+b_i^\top\mu_i+d_i^\top\nu_i,\\
\text{s.t.}\;&v_i=y_i-q_i-A_i^\top\mu_i-E_i^\top\nu_i .
\end{aligned}
\label{eq:conjugate_qp_dual}
\end{equation}
Thus, the SCA model may include \eqref{eq:conjugate_qp_dual} as a convex epigraph. 
If \(Q_i\in\Spsd^{n_i}\), the same construction is applied after the small regularization used for response uniqueness, or through the corresponding closed conic epigraph when available.

\subsection{Relation to KKT and MPEC Reformulations}
\label{subsec:relation_to_kkt}

If \(\X_i\) is polyhedral as in \eqref{eq:polyhedral_Xi}, a conventional KKT reformulation of \eqref{eq:follower_qp} introduces multipliers \(\mu_i\ge0\) and \(\nu_i\) and imposes
\begin{subequations}
\begin{align}
    &A_i x_i\le b_i,
    \qquad E_i x_i=d_i,\label{eq:kkt_primal}\\
    &Q_i x_i+q_i+\lambda_i(\delta_i)+A_i^\top\mu_i+E_i^\top\nu_i=0,\label{eq:kkt_stationarity}\\
    &\mu_i\ge0,\qquad
    \mu_i^\top(b_i-A_i x_i)=0.
    \label{eq:kkt_complementarity}
\end{align}
\label{eq:kkt_system}
\end{subequations}
The complementarity equation \eqref{eq:kkt_complementarity} is the source of nonconvexity. In mixed-integer reformulations, each complementarity pair is commonly encoded through Fortuny--Amat big-\(M\) constraints, introducing binary variables and requiring valid bounds on dual variables and slacks.

The FY reformulation avoids explicit complementarity. It represents lower-level optimality by the scalar gap \(D_{\phi_i}\). If \(\phi_i^*\) has a closed form, no lower-level dual variables are needed. If an explicit convex epigraph of \(\phi_i^*\) is constructed through duality, auxiliary continuous variables may appear, but complementarity and big-\(M\) disjunctions are still avoided. This distinction is important in high-dimensional energy-hub models, and, in fact, the proposed method shifts the computational burden from binary complementarity enforcement to continuous convex approximation.


\subsection{Sequential Convex Approximation}
\label{sec:sca}

\subsubsection{Penalized single-level reformulation}
\label{subsec:penalized_reformulation}

Rather than imposing \(D_{\phi_i}=0\) as a hard equality, define the penalized objective
\begin{equation}
    \mathcal{P}_\rho(x,\delta)=F(x,\delta)+\rho\sum_{i\in\Hubs}D_{\phi_i}(x_i,-\lambda_i(\delta_i)),
    \label{eq:penalized_objective}
\end{equation}
where \(\rho>0\) is a penalty parameter. The penalized problem is
\begin{equation}
    \min_{x,\delta}\;\mathcal{P}_\rho(x,\delta)
    \quad \text{s.t.}\quad \delta\in\D .
    \label{eq:penalized_problem}
\end{equation}
The penalty is nonnegative and vanishes when the follower responses are consistent with the selected prices. A large penalty discourages price-response inconsistency. Exact finite-penalty recovery of \eqref{eq:single_level_fy_exact} requires additional error-bound or calmness conditions; the algorithm below is therefore stated as an SCA method for the penalized problem, not as an unconditional global solver for the original bilevel program.

\subsubsection{Difference-of-convex decomposition}
\label{subsec:dc_decomposition}

For notational clarity, first consider a follower with price coefficient \(\lambda_i\) of the same dimension as \(x_i\). The FY gap is
\begin{equation}
D_{\phi_i}(x_i,-\lambda_i)=\phi_i(x_i)+\phi_i^*(-\lambda_i)+\lambda_i^\top x_i.
\end{equation}
Using the polarization identity,
\begin{equation}
    \lambda_i^\top x_i
    =\frac14\|x_i+\lambda_i\|^2-\frac14\|x_i-\lambda_i\|^2,
    \label{eq:polarization}
\end{equation}
we write
\begin{equation}
    D_{\phi_i}(x_i,-\lambda_i)=h_i(x_i,\lambda_i)-o_i(x_i,\lambda_i),
    \label{eq:dc_gap}
\end{equation}
where
\begin{subequations}
\begin{align}
    h_i(x_i,\lambda_i)&=\phi_i(x_i)+\phi_i^*(-\lambda_i)+\frac14\|x_i+\lambda_i\|^2,\label{eq:h_i}\\
    o_i(x_i,\lambda_i)&=\frac14\|x_i-\lambda_i\|^2.
    \label{eq:o_i}
\end{align}
\end{subequations}
Both \(h_i\) and \(o_i\) are convex. At iteration \(k\), the concave term \(-o_i\) is upper-bounded by its first-order affine decomposition at \((x_i^k,\lambda_i^k)\):
\begin{equation}
\begin{aligned}
    \widehat D_i^k(x_i,\lambda_i)=&\;h_i(x_i,\lambda_i)-o_i(x_i^k,\lambda_i^k)\\
    &-\left\langle \nabla o_i(x_i^k,\lambda_i^k),\col(x_i-x_i^k,\lambda_i-\lambda_i^k)\right\rangle .
\end{aligned}
\label{eq:majorized_gap}
\end{equation}
Equivalently, after dropping constants that do not affect the minimizer,
\begin{align}
\widehat D_i^k(x_i,\lambda_i)
\equiv{}&\phi_i(x_i)+\phi_i^*(-\lambda_i)
+\frac14\|x_i+\lambda_i\|^2 \nonumber\\
&-\frac12(x_i^k-\lambda_i^k)^\top(x_i-\lambda_i).
\label{eq:majorized_gap_no_constants}
\end{align}

In the distribution-pricing application, the actual price coefficient is \(\lambda_i(\delta_i)\). The SCA subproblem is obtained by replacing \(\lambda_i\) in \eqref{eq:majorized_gap} with the affine mapping \(\lambda_i(\delta_i)\). Since composition with an affine mapping preserves convexity, the resulting subproblem remains convex under Assumption~\ref{ass:leader_convex}.

\subsubsection{Algorithmic implementation}
\label{subsec:algorithm}

\begin{algorithm}[!t]
\caption{Fenchel--Young SCA for dynamic congestion pricing}
\label{alg:fy_sca}
\begin{algorithmic}[1]
\REQUIRE Initial price \(\delta^0\in\D\), response \(x^0\), penalty \(\rho_0>0\), penalty cap \(\rho_{\max}\ge\rho_0\), update factor \(\gamma_\rho>1\), stall threshold \(\eta_\rho\in(0,1)\), tolerances \(\varepsilon_{\rm gap},\varepsilon_{\rm step}\), and \(K_{\max}\).
\ENSURE Dynamic price \(\delta^\star\), response \(x^\star\), FY residuals, and convergence diagnostics.
\STATE Compute \(\lambda_i^0=\lambda_i(\delta_i^0)\), \(r_i^0=D_{\phi_i}(x_i^0,-\lambda_i^0)\), and \(r_{\max}^0=\max_{i\in\Hubs}r_i^0\).
\FOR{\(k=0,1,\ldots,K_{\max}-1\)}
    \STATE Build the convex majorized gaps \(\widehat D_i^k(x_i,\lambda_i(\delta_i))\) using \eqref{eq:majorized_gap}.
    \STATE Solve the convex SCA subproblem
    \begin{equation*}
    \begin{aligned}
    (x^{k+1},\delta^{k+1})\in\argmin_{x,\delta}\
    &F(x,\delta)+\rho_k\sum_{i\in\Hubs}\widehat D_i^k(x_i,\lambda_i(\delta_i))\\
    \text{s.t.}\quad &\delta\in\D.
    \end{aligned}
    \end{equation*}
    using an explicit epigraph or oracle for \(\phi_i^*\).
    \STATE Compute \(\lambda_i^{k+1}=\lambda_i(\delta_i^{k+1})\) and residuals
    \begin{equation*}
        r_i^{k+1}=D_{\phi_i}(x_i^{k+1},-\lambda_i^{k+1}),\quad
        r_{\max}^{k+1}=\max_{i\in\Hubs}r_i^{k+1}.
    \end{equation*}
    \STATE Compute \(s^{k+1}=\|\delta^{k+1}-\delta^k\|_2+\sum_i\|x_i^{k+1}-x_i^k\|_2\).
    \IF{\(r_{\max}^{k+1}\le\varepsilon_{\rm gap}\) and \(s^{k+1}\le\varepsilon_{\rm step}\)}
        \STATE \textbf{break}
    \ENDIF
    \IF{\(k\ge1\), \(r_{\max}^{k+1}>\varepsilon_{\rm gap}\), and \(r_{\max}^{k+1}>\eta_\rho r_{\max}^{k}\)}
    \STATE Set \(\rho_{k+1}=\min\{\gamma_\rho\rho_k,\rho_{\max}\}\).
\ELSE
    \STATE Set \(\rho_{k+1}=\rho_k\).
\ENDIF
\ENDFOR
\STATE Return the best iterate according to the original penalized objective and report all residuals and network metrics.
\end{algorithmic}
\end{algorithm}


Algorithm~\ref{alg:fy_sca} summarizes the offline FY-SCA procedure used by the DSO to compute dynamic congestion-price adders. The FY gap and its convex decomposition are defined at the hub-horizon level, because \(x_i\) collects the full multi-period response of hub \(i\). Accordingly, the penalty term is summed over hubs, not over independent period-wise gaps.

The penalty update is deliberately conservative. The parameter \(\eta_\rho\) declares a stall when the maximum FY residual does not decrease by the prescribed fraction. In that case, increasing \(\rho_k\) shifts weight toward price-response consistency, while the cap \(\rho_{\max}\) limits ill-conditioning of the convex SCA subproblems. When the residual decreases sufficiently, \(\rho_k\) is held fixed.

The algorithm is used for offline price design with forecasts of demand, renewable availability, network states, and calibrated aggregate hub-response information. Once the price adders are selected, implementation is decentralized: the DSO broadcasts \(\delta_i\), or the induced coefficient \(\lambda_i(\delta_i)\), and each hub solves its own convex follower problem. Thus, the final operation is induced by prices rather than imposed through direct dispatch. {\color{black}The computational workflow is summarized in Fig.~B.1}.

\subsubsection{Convergence and stationarity}
\label{subsec:convergence}

\begin{proposition}[Convexity of each SCA subproblem]
\label{prop:convex_subproblem}
Under Assumptions~\ref{ass:convex_price_response}--\ref{ass:leader_convex}, if \(F\) is convex in the explicit variables of the SCA subproblem and \(\lambda_i(\delta_i)\) is affine, then each subproblem in Algorithm~\ref{alg:fy_sca} is convex.
\end{proposition}

\begin{proof}
The terms \(\phi_i(x_i)\), \(\phi_i^*(-\lambda_i(\delta_i))\), and \(\frac14\|x_i+\lambda_i(\delta_i)\|^2\) are convex because \(\phi_i\) and \(\phi_i^*\) are convex and \(\lambda_i(\delta_i)\) is affine. The linearized term in \eqref{eq:majorized_gap_no_constants} is affine. Hence \(\widehat D_i^k\) is convex. Adding convex \(F\) and imposing convex \(\D\) preserves convexity.
\end{proof}

\begin{proposition}[Stationarity of limit points]
\label{prop:stationarity}
Suppose the level set of the penalized objective is compact, the convex subproblems are solved exactly, \(\rho_k=\rho\) is fixed, and standard regularity conditions for the DC algorithm hold. Then every accumulation point of the sequence generated by Algorithm~\ref{alg:fy_sca} is a critical point of the penalized problem \eqref{eq:penalized_problem}.
\end{proposition}

\begin{proof}
The majorized gap \(\widehat D_i^k\) is a global upper bound of \(D_{\phi_i}\) and is tight at \((x_i^k,\lambda_i^k)\). Therefore the SCA step is an instance of the convex-concave procedure or DC algorithm. Under compactness and regularity, the standard descent and closedness arguments for DC programming imply that every accumulation point satisfies the first-order criticality condition of the penalized DC problem \cite{tao1997dc}.
\end{proof}

\begin{remark}[Characterization of solution guarantees]
The proposed SCA method is not a global solver for arbitrary bilevel programs; rather, it generates a sequence of convex approximations to a penalized exact reformulation. Under appropriate penalty updates and assumptions, the method drives the FY residual toward zero, converging to a stationary point of the penalized problem. 
\end{remark}

\subsubsection{Computational complexity discussion}
\label{subsec:complexity}

A KKT/MPEC reformulation introduces stationarity, dual-feasibility, and complementarity conditions for each follower problem. As in mixed-integer programming formulations, the number of binary variables scales with the number of follower inequality constraints. In multi-period energy hub models, this number may reach hundreds or thousands, resulting in large-scale mixed-integer linear programs (MILPs) that are difficult to solve due to their inherent NP-hardness, even before distribution-network constraints are incorporated.
This observation is consistent with complexity results for multi-period tariff optimization and bilevel DR models \cite{kovacs2018complexity,kovacs2019bilevel}. In contrast, the FY-SCA formulation avoids binary complementarity. Its continuous convex subproblems may still be large, especially if \(\phi_i^*\) is represented through conic or dual epigraphs; however, the computational bottleneck is shifted from mixed-integer combinatorics to convex optimization. This is attractive when the price dimension is low or structured, such as one price adder per hub and time, while the follower model contains many device-level constraints.

\section{Distribution-Network Model and Benchmark Design}
\label{sec:congestion_specialization}
Sections~\ref{sec:preliminaries}--\ref{sec:fy_reformulation} develop the more generic machinery for dynamic pricing. We now tailor it to the radial-feeder DSO--hub congestion-pricing application used in the numerical studies. Fig.~\ref{fig:section_iv_compact_system} summarizes the process: the DSO broadcasts a price adder, the hubs solve local problems, and their net withdrawals determine the monitored feeder flows and the congestion status.

\begin{figure}[!ht]
\centering
\resizebox{\columnwidth}{!}{%
\begin{tikzpicture}[
    x=1cm,y=1cm,
    font=\scriptsize,
    >=Latex,
    box/.style={
        draw=black!78,
        rounded corners=1.2pt,
        align=center,
        fill=black!2,
        inner sep=2pt,
        line width=0.42pt
    },
    dso/.style={box,minimum width=35mm,minimum height=9mm},
    hub/.style={box,minimum width=17mm,minimum height=8.5mm},
    bus/.style={circle,fill=black,inner sep=1.25pt},
    feeder/.style={line width=0.70pt,black!88},
    phy/.style={
        {Latex[length=1.2mm,width=0.9mm]}-{Latex[length=1.2mm,width=0.9mm]},
        line width=0.50pt,
        black!75
    },
    infoline/.style={
        line width=0.52pt,
        blue!70!black
    },
    infoarrow/.style={
        -{Latex[length=1.25mm,width=0.95mm]},
        line width=0.52pt,
        blue!70!black
    }
]

\def\yDSO{3.10}
\def\yFeeder{2.22}
\def\yHub{1.28}
\def\yInfo{0.42}

\def\xOne{-2.40}
\def\xTwo{-0.55}
\def\xHubDots{1.00}
\def\xFeederDots{1.00}
\def\xI{2.40}

\def\xInfoLeft{-3.55}

\node[dso] (dso) at (0,\yDSO) {\textbf{DSO}\\[-0.25ex]
{\tiny Forecasts, $H_{\ell i}$, limits, tariffs}\\[-0.25ex]
{\tiny compute common $\delta_t$}};

\node[bus] (b1) at (\xOne,\yFeeder) {};
\node[bus] (b2) at (\xTwo,\yFeeder) {};
\node[bus] (bi) at (\xI,\yFeeder) {};

\draw[feeder] (b1.center) -- (bi.center);
\draw[feeder] (0,\yDSO-0.48) -- (0,\yFeeder);

\node[above=1.2pt,font=\tiny] at (\xOne,\yFeeder) {1};
\node[above=1.2pt,font=\tiny] at (\xTwo,\yFeeder) {2};
\node[above=1.2pt,font=\tiny] at (\xI,\yFeeder) {$i$};

\node[fill=white,inner sep=0.35pt] at (\xFeederDots,\yFeeder) {$\cdots$};

\node[hub] (h1) at (\xOne,\yHub) {Hub 1\\[-0.25ex]$x_1\in\X_1$};
\node[hub] (h2) at (\xTwo,\yHub) {Hub 2\\[-0.25ex]$x_2\in\X_2$};
\node[hub] (hi) at (\xI,\yHub) {Hub $i$\\[-0.25ex]$x_i\in\X_i$};

\node at (\xHubDots,\yHub) {$\cdots$};

\draw[phy] (h1.north) -- node[left=0.45mm] {$w_{1t}$} (b1.south);
\draw[phy] (h2.north) -- node[right=0.45mm] {$w_{2t}$} (b2.south);
\draw[phy] (hi.north) -- node[right=0.45mm] {$w_{it}$} (bi.south);





\coordinate (dsoOut) at (dso.west);

\draw[infoline] (dsoOut) -- (\xInfoLeft,\yDSO) -- (\xInfoLeft,\yInfo) -- (0.55,\yInfo);
\draw[infoline] (1.45,\yInfo) -- (\xI,\yInfo);
\node[fill=white,inner sep=0.35pt,text=blue!70!black] at (1.00,\yInfo) {$\cdots$};

\draw[infoarrow] (\xOne,\yInfo) -- (h1.south);
\draw[infoarrow] (\xTwo,\yInfo) -- (h2.south);
\draw[infoarrow] (\xI,\yInfo)   -- (hi.south);

\node[text=blue!70!black,font=\tiny,anchor=south]
at (-1.10,\yInfo+0.03) {Broadcast $\delta_t$};
\draw[infoarrow] (bi.east) -- ++(0.75,0) -- ++(0,\yDSO-\yFeeder) -- (dso.east);

\node[text=blue!70!black,font=\tiny,anchor=east]
at (\xI+0.68,2.89) {Network state};
\draw[feeder] (-3.45,0.08) -- (-3.12,0.08);
\node[font=\tiny,anchor=west] at (-3.05,0.08) {Power flow line};

\draw[infoline] (-1.32,0.08) -- (-0.99,0.08);
\node[font=\tiny,text=blue!70!black,anchor=west] at (-0.92,0.08) {Information flow line};
\end{tikzpicture}}
\caption{Compact system model of the distribution-network application. The DSO computes the common congestion-price adder \(\delta_t\); hubs solve local problems over \(\X_i\); the resulting net withdrawals \(w_{it}\) determine monitored feeder flows, and network-state information is fed back to the DSO.}
\label{fig:section_iv_compact_system}
\vspace{-0.8ex}
\end{figure}
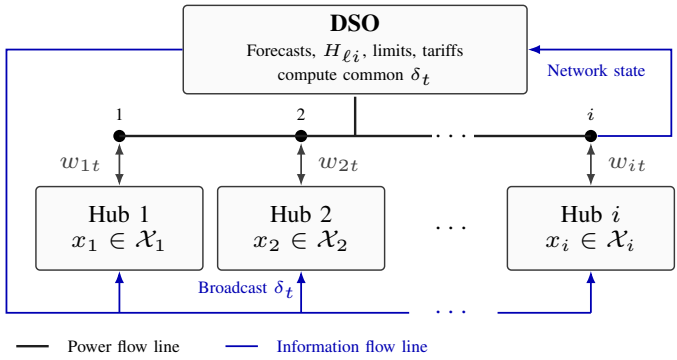

\subsection{Convex Network and Congestion Model}
\label{subsec:network_model}

For hub \(i\) and period \(t\), net withdrawal is
\begin{equation}
    w_{it}=m^{\rm in}_{it}-m^{\rm out}_{it}.
    \label{eq:net_withdrawal}
\end{equation}
Here \(m^{\rm in}_{it}\) and \(m^{\rm out}_{it}\) are the hub's grid import and export, respectively, so positive \(w_{it}\) denotes net withdrawal from the feeder. The numerical study uses a PTDF-type positive-sequence active-power sensitivity abstraction,
\begin{equation}
    f_{\ell t}=f^{\rm bg}_{\ell t}-\sum_{i\in\Hubs}H_{\ell i}w_{it},
    \qquad \ell\in\Lines,
    \label{eq:ptdf_flow}
\end{equation}
and \(S_t=S_t^{\rm bg}+\sum_i w_{it}\) for the substation exchange. In \eqref{eq:ptdf_flow}, \(f_{\ell t}\) and \(f_{\ell t}^{\rm bg}\) denote the monitored and background active-power flow on line \(\ell\), \(H_{\ell i}\) is the corresponding PTDF-type sensitivity to hub \(i\), and \(S_t^{\rm bg}\) is the background substation exchange. This is our modeling choice. Distribution-pricing studies have likewise used PTDF-based active-power congestion models or linearized distribution OPF models to obtain tractable congestion prices \cite{li2014dlmp_ev,huang2015dlmp_qp,yuan2018linearized}. The choice is not imposed by the FY reformulation or by polyhedral follower constraints: any convex upper-level network representation, including linearized DistFlow or an SOCP/SDP relaxation, may replace \eqref{eq:ptdf_flow} while preserving convex SCA subproblems \cite{farivar2013branchflow,lavaei2012zero}. 

Line and substation overloads are represented by
\begin{subequations}
\begin{align}
 o^{\rm line}_{\ell t}&\ge f_{\ell t}-\bar f_\ell,&
 o^{\rm line}_{\ell t}&\ge-f_{\ell t}-\bar f_\ell,&
 o^{\rm line}_{\ell t}&\ge0,\label{eq:line_overload}\\
 o^{\rm sub}_{t}&\ge S_t-\bar S^+,&
 o^{\rm sub}_{t}&\ge-S_t-\bar S^-,&
 o^{\rm sub}_{t}&\ge0.\label{eq:sub_overload}
\end{align}
\label{eq:overload_epigraphs}
\end{subequations}
The variables \(o^{\rm line}_{\ell t}\) and \(o^{\rm sub}_{t}\) are nonnegative epigraph residuals for violations of line limit \(\bar f_\ell\) and substation import/export limits \(\bar S^+,\bar S^-\). The DSO congestion term is
\begin{equation}
 J^{\rm cong}(x)=\alpha^{\rm line}\!\sum_{\ell,t}\Gamma(o^{\rm line}_{\ell t})
 +\alpha^{\rm sub}\!\sum_t\Gamma(o^{\rm sub}_{t}),
 \label{eq:congestion_objective}
\end{equation}
where \(\Gamma\) is convex and nondecreasing. The coefficients \(\alpha^{\rm line}\) and \(\alpha^{\rm sub}\) weight line and substation violations in the DSO objective. The reported physical metrics use the same epigraph residuals without objective weights:
\begin{align}
 C^{\rm line}&=\sum_{\ell,t}o^{\rm line}_{\ell t},\qquad
 C^{\rm sub}=\sum_t o^{\rm sub}_{t},\notag\\[-0.2ex]
 C^{\rm total}&=C^{\rm line}+C^{\rm sub}.
 \label{eq:congestion_metrics}
\end{align}

\subsection{Participant Response and Dynamic Price Design}
\label{subsec:energy_hub_model}

The term \emph{hub} denotes a controllable boundary node, not a restriction on participation. It may represent a VPP, DER aggregator, commercial site, EV fleet, or portfolio of conventional customers. Inelastic customers enter \(D^{\rm fix}_{it}\) or \(f^{\rm bg}_{\ell t}\); price-responsive conventional customers are simplified hubs containing only import and flexible-load variables.

For each \((i,t)\), let
\begin{equation}
 x_{it}=(m^{\rm in}_{it},m^{\rm out}_{it},u_{it},p^{\rm ch}_{it},p^{\rm dis}_{it},e_{it},a_{it},c^{\rm pv}_{it},g_{it}).
 \label{eq:xit_order}
\end{equation}
Thus \(x_{it}\) collects import/export, flexible demand \(u_{it}\), battery charging/discharging and energy state \((p^{\rm ch}_{it},p^{\rm dis}_{it},e_{it})\), the absolute flexible-load deviation \(a_{it}\), PV curtailment \(c^{\rm pv}_{it}\), and dispatchable generation \(g_{it}\). The power balance is
\begin{equation}
\begin{aligned}
&m^{\rm in}_{it}-m^{\rm out}_{it}+p^{\rm dis}_{it}-p^{\rm ch}_{it}+g_{it}-u_{it}-c^{\rm pv}_{it}\\
&\hspace{27mm}=D^{\rm fix}_{it}-PV^{\rm avail}_{it}.
\end{aligned}
\label{eq:hub_power_balance}
\end{equation}
Here \(D^{\rm fix}_{it}\) is fixed demand and \(PV^{\rm avail}_{it}\) is available PV generation before curtailment. Flexible demand is bounded and energy conserving,
\begin{subequations}
\begin{align}
(1-\alpha_i^{\rm down})u^0_{it}&\le u_{it}\le(1+\alpha_i^{\rm up})u^0_{it},\label{eq:flex_bounds}\\
\sum_tu_{it}&=\sum_tu^0_{it},\label{eq:flex_energy}\\
a_{it}&\ge \pm(u_{it}-u^0_{it}),\label{eq:flex_abs}
\end{align}
\label{eq:flex_constraints}
\end{subequations}
In \eqref{eq:flex_constraints}, \(u^0_{it}\) is the preferred flexible-load baseline and \(\alpha_i^{\rm down},\alpha_i^{\rm up}\) define downward/upward adjustment limits, and storage satisfies
\begin{subequations}
\begin{align}
e_{it}&=e_{i,t-1}+\eta_i^{\rm ch}p^{\rm ch}_{it}\dd-(\eta_i^{\rm dis})^{-1}p^{\rm dis}_{it}\dd,\label{eq:battery_dyn}\\
\underline e_i&\le e_{it}\le\bar e_i,\quad
0\le p^{\rm ch}_{it}\le\bar p_i^{\rm ch},\quad
0\le p^{\rm dis}_{it}\le\bar p_i^{\rm dis},\label{eq:battery_bounds}\\
p^{\rm ch}_{it}+p^{\rm dis}_{it}&\le\bar p^{\rm bat}_{it}.
\label{eq:battery_antiloop}
\end{align}
\label{eq:battery_constraints}
\end{subequations}
In \eqref{eq:battery_constraints}, \(\eta_i^{\rm ch}\) and \(\eta_i^{\rm dis}\) are charging/discharging efficiencies and the barred quantities are device limits. Together with import/export, curtailment, and dispatchable-generation bounds, these constraints define \(\X_i\).

The effective prices are
\begin{equation}
 \pi^{\rm buy}_{it}=\lambda_t^{\rm buy}+\tau_i^{\rm base}+\delta_{it},\qquad
 \pi^{\rm sell}_{it}=\lambda_t^{\rm sell}+\chi^{\rm sell}\delta_{it}.
 \label{eq:effective_prices}
\end{equation}
Here \(\lambda_t^{\rm buy}\) and \(\lambda_t^{\rm sell}\) are reference buy/sell prices, \(\tau_i^{\rm base}\) is the hub-specific base tariff, and \(\chi^{\rm sell}\) passes the congestion adder through to export credit. Thus \(\chi^{\rm sell}=0\) gives an import-only surcharge, whereas \(\chi^{\rm sell}=1\) gives symmetric net-withdrawal pricing. The IEEE-feeder experiments use \(\chi^{\rm sell}=1\), as in the implemented market design. 

Let \(M_i^{\rm in}\) and \(M_i^{\rm out}\) be selection matrices extracting imports and exports from \(x_i\). The price-induced linear coefficient in \eqref{eq:follower_qp} is $\lambda_i = \col(\lambda_{it}:t\in\mathcal{T})$, with \(\lambda_{it}(\delta_{it})\) denoting the period-\(t\) affine price coefficient applied inside the hub objective, where
\begin{equation}
\begin{aligned}
    \lambda_{it}(\delta_{it})=\dd\big((M_i^{\rm in})^\top(\lambda_t^{\rm buy}+\tau_i^{\rm base}\ones+\delta_{it})\\
    -(M_i^{\rm out})^\top(\lambda_t^{\rm sell}+\chi^{\rm sell}\delta_{it})\big).
    \label{eq:lambda_i_energy}
    \end{aligned}
\end{equation}
This definition keeps the follower objective affine in the DSO decision \(\delta_{it}\). In \eqref{eq:lambda_i_energy}, \(\ones\) denotes a conformable vector of ones and \(\dd\) is the settlement-period length. However, it should also be mentioned that the main experiments impose the common broadcast policy $\delta_{it}=\delta_t$, $\forall i,t$, with caps $-40\le\delta_t\le80$ EUR/MWh. 
\subsection{Dynamic Congestion Price Design}
\label{subsec:price_design}
The upper-level objective used in the experiments is
\begin{equation}
    F(x,\delta)=J^{\rm cong}(x)+J^{\rm op}(x)+R(\delta),
    \label{eq:application_F}
\end{equation}
where \(J^{\rm cong}(x)\) penalizes network congestion, \(J^{\rm op}(x)\) captures any convex operating-cost or welfare term, and \(R(\delta)\) regularizes the dynamic congestion-price adders. In the DSO case, \(J^{\rm op}\) may include losses, curtailment, balancing costs, or welfare losses; otherwise, it can be set to zero without affecting the proposed algorithm. Price regularization is
\begin{equation}
\begin{aligned}
    R(\delta)=&\;\beta^{\rm level}\sum_{t\in\Time}\sum_{i\in\Hubs}\delta_{it}^2
    +\beta^{\rm fair}\sum_{t\in\Time}\sum_{i\in\Hubs}(\delta_{it}-\bar\delta_t)^2,
\end{aligned}
\label{eq:price_regularization}
\end{equation}
with \(\bar\delta_t=|\Hubs|^{-1}\sum_i\delta_{it}\). The first term discourages unnecessarily large price adders, and the second discourages excessive cross-hub price dispersion. After the DSO computes $\lambda_{i}$, each hub solves its convex follower problem and reports $w_{it}$ and settlement quantities.

The final application-level penalized SCA problem at iteration \(k\) is 
\begin{equation}
\begin{aligned}
\min_{x,\delta,o}\quad
& F(x,\delta)+\rho_k\sum_{i\in\Hubs}\widehat D_i^k(x_i,\lambda_i(\delta_i))\\
\text{s.t.}\quad
&\delta\in\D,\quad
\eqref{eq:ptdf_flow}-\eqref{eq:overload_epigraphs},\quad x_i\in\X_i.
\end{aligned}
\label{eq:application_sca}
\end{equation}

\textit{Implementation and information exchange}:
The proposed framework separates offline price design from real-time price implementation. In the numerical study, the DSO is assumed to have forecasted demand, renewable availability, network states, and a calibrated aggregate model, or oracle, of each hub's convex price response. This information is used only to compute the dynamic congestion-price adders by solving \eqref{eq:application_sca}; it does not imply direct dispatch of private hub devices. During implementation, the DSO broadcasts the computed adder \(\delta_{it}\), or equivalently the induced price coefficient \(\lambda_{it}(\delta_{it})\), to each hub. Each hub then solves \eqref{eq:follower_qp} locally and reports only the quantities required for settlement and network monitoring, such as net withdrawal \(w_{it}\) and delivered flexibility. Thus, dynamic prices act as coordination signals rather than direct dispatch commands.
Load-response metrics include
\begin{equation}
E^{\rm flex,shift}=\sum_{i,t}|u_{it}-u^0_{it}|\dd,
\quad
E^{\rm bat,thr}=\sum_{i,t}(p^{\rm ch}_{it}+p^{\rm dis}_{it})\dd,
\label{eq:flex_metrics}
\end{equation}
where \(E^{\rm flex,shift}\) measures the total flexible-load adjustment relative to the preferred baseline \(u^0_{it}\), and \(E^{\rm bat,thr}\) measures total battery cycling. The lower-level optimality certificate for hub \(i\) is
$r_i
=
D_{\phi_i}\!\left(x_i,-\lambda_i(\delta_i)\right)$,
with \(r_i=0\) indicating the consistency between the broadcast price and the hub's convex response model.

\subsection{IEEE Feeder Construction and Evaluation Protocol}
\label{subsec:benchmark_protocol}


The physical topologies and static spot-load maps are taken from the standard IEEE 13-node and 34-node radial distribution test feeders \cite{kersting2001radial}. The study preserves feeder connectivity and active-load geography but uses the convex sensitivity abstraction in \eqref{eq:ptdf_flow}. Eight loaded buses are promoted to active aggregators. For each original static load, 85\% is assigned to the electrically nearest active hub and 15\% remains at its original bus as passive background demand. Flexible demand, PV, batteries, and modest dispatchable DER are then added at the hubs using transparent fractions of assigned load. 
Thus the benchmark topology and base load are standard IEEE data, whereas the time profiles, active-resource portfolios, and stress ratings are documented study augmentations. Fig.~\ref{fig:ieee13_schematic} shows a high-level schematic of the augmented IEEE 13-node case, while {\color{black}the complete augmentation rules and their details are summarized in Table~C.1.}

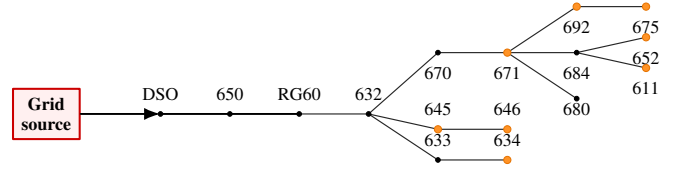
\begin{figure}[!t]
\centering
\resizebox{0.98\columnwidth}{!}{%
\begin{tikzpicture}[
  x=0.950cm,y=0.420cm,>=Latex,line cap=round,line join=round,
  feeder/.style={line width=0.45pt,draw=black!85},trunk/.style={line width=0.75pt,draw=black},
  bus/.style={circle,draw=black,fill=black,inner sep=0pt,minimum size=1.8pt},
  source/.style={draw=red!80!black,line width=0.8pt,fill=red!6,minimum width=2.3em,minimum height=1.5em,align=center,font=\scriptsize\bfseries}
]
\coordinate (bDSO) at (0,-1);\coordinate (b650) at (1,-1);\coordinate (bRG60) at (2,-1);\coordinate (b632) at (3,-1);
\coordinate (b633) at (4,-2.5);\coordinate (b645) at (4,-1.5);\coordinate (b670) at (4,1);
\coordinate (b634) at (5,-2.5);\coordinate (b646) at (5,-1.5);\coordinate (b671) at (5,1);
\coordinate (b680) at (6,-.5);\coordinate (b684) at (6,1);\coordinate (b692) at (6,2.5);
\coordinate (b611) at (7,.5);\coordinate (b652) at (7,1.5);\coordinate (b675) at (7,2.5);
\node[source,anchor=east] (source) at (-1.15,-1) {Grid\\source};\draw[trunk,-Latex] (source.east)--(bDSO);
\foreach \a/\b in {bDSO/b650,b650/bRG60}{\draw[trunk] (\a)--(\b);}
\foreach \a/\b in {bRG60/b632,b632/b670,b632/b633,b632/b645,b670/b671,b633/b634,b645/b646,b671/b680,b671/b684,b671/b692,b684/b611,b684/b652,b692/b675}{\draw[feeder] (\a)--(\b);}
\foreach \n/\lab/\x/\y/\anc in {bDSO/DSO/0/-.84/south,b650/650/1/-.84/south,bRG60/RG60/2/-.84/south,b632/632/3/-.84/south,b633/633/4/-2.34/south,b670/670/4/.84/north,b680/680/6/-.34/north,b684/684/6/.84/north}{\node[bus] at (\n){};\node[font=\scriptsize,anchor=\anc] at (\x,\y){\lab};}
\foreach \n/\lab/\x/\y/\anc in {b645/645/4/-1.34/south,b634/634/5/-2.34/south,b646/646/5/-1.34/south,b671/671/5/.84/north,b692/692/6/2.34/north,b611/611/7/.34/north,b652/652/7/1.34/north,b675/675/7/2.34/north}{\node[bus,draw=orange!90!black,fill=orange!85,minimum size=2.8pt] at (\n){};\node[font=\scriptsize,anchor=\anc] at (\x,\y){\lab};}
\end{tikzpicture}}
\vspace{-0.7ex}
\caption{Equivalent schematic of the augmented IEEE 13-node case. Orange buses host the eight controllable aggregators; black buses and feeder connectivity follow the benchmark topology. }
\label{fig:ieee13_schematic}
\end{figure}

Both cases use \(T=24\), \(\dd=1\) h, flexible-load ranges up to \(\pm35\%\), symmetric buy/sell adders, and one common \(\delta_t\) across hubs. The four benchmarks are: Base (no dynamic price), Central (full-information physical lower bound), KKT/MILP (the conventional linear-follower complementarity benchmark), and Quad-FY (the proposed model). KKT/MILP is subject to a 3600-s limit; when that limit is reached, its row is an incumbent rather than a proven optimum. Central is not an implementable market mechanism, and it only acts as a utopian approach that has full knowledge and control of the hubs. Calculations were performed on a 64-bit Windows 10 workstation with an Intel Core i9-13900K CPU and 16 GB RAM.

\section{Numerical Results}
\label{sec:numerical_results}

\subsection{Congestion Mitigation and Temporal Performance}

Table~\ref{tab:real_benchmark_results} consolidates the physical, economic, and computational outcomes. Quad-FY reduces total congestion from 16.79 to 0.522 MW on IEEE 13 and from 35.58 to 1.263 MW on IEEE 34, corresponding to 96.89\% and 96.45\% reductions. 
Central eliminates the overload in both cases, showing that the remaining FY residual is not an irreducible network minimum. It instead reflects the restricted common-price instrument, decentralized quadratic responses, and local outer optimization.

\begin{table*}[!t]
\centering
\caption{IEEE Feeder Benchmark Outcomes and Price-Response Diagnostics}
\label{tab:real_benchmark_results}
\footnotesize
\setlength{\tabcolsep}{2.35pt}\renewcommand{\arraystretch}{0.92}
\begin{tabular}{llrrrrrrrrrL{0.16\textwidth}}
\toprule
Feeder & Method & Cong. & Red. & FY Iter. & Time & $|\delta|$ & Ramp & Flex & Batt. & Import & Certificate/status\\
 & & (MW) & (\%) & & (s) & \multicolumn{2}{c}{(EUR/MWh)} & \multicolumn{3}{c}{(MWh)} & \\
\midrule
IEEE 13 & Base & 16.786 & 0.00 & 0 & 0.28 & 0 & 0 & 2.971 & 3.926 & 61.70 & LP/QP optimal\\
 & Central & 0 & 100 & 0 & 0.74 & 0 & 0 & 3.203 & 3.978 & 57.51 & physical lower bound\\
 & KKT/MILP & 2.614 & 84.43 & -- & 3606 & 33.87 & 36.81 & 3.408 & 3.926 & 54.76 & 3600-s limit; incumbent\\
 & Quad-FY & 0.522 & 96.89 & 5 & 732.1 & 45.69 & 32.40 & 3.951 & 3.926 & 54.00 & $r_{\max}=3.41\!\times\!10^{-13}$\\
\midrule
IEEE 34 & Base & 35.579 & 0.00 & 0 & 0.30 & 0 & 0 & 1.452 & 2.145 & 30.20 & LP/QP optimal\\
 & Central & 0 & 100 & 0 & 1.99 & 0 & 0 & 2.001 & 2.298 & 29.11 & physical lower bound\\
 & KKT/MILP & 4.377 & 87.70 & -- & 3609 & 50.81 & 45.25 & 2.258 & 2.145 & 26.77 & 3600-s limit; incumbent\\
 & Quad-FY & 1.263 & 96.45 & 7 & 882.5 & 54.51 & 29.17 & 2.204 & 2.145 & 26.77 & $r_{\max}=1.71\!\times\!10^{-13}$\\
\bottomrule
\end{tabular}
\vspace{0.4ex}
\begin{minipage}{1.5\columnwidth}
\footnotesize
\emph{Note:} The iteration count is reported only for Quad-FY and denotes outer FY-SCA decomposition iterations. The KKT/MILP benchmark is solved as a single mixed-integer optimization problem.
\end{minipage}
\end{table*}

The 34-node baseline is more severe, yet FY retains essentially the same percentage reduction. By contrast, KKT/MILP leaves 2.614 and 4.377 MW and reaches the one-hour limit in both cases. The measured KKT/FY runtime ratios are at least 4.93 and 4.09, respectively; because KKT did not close its solve, these are lower bounds on the time required for a proven KKT solution.  Increasing the processed network from 16 nodes/18 lines to 56 nodes/60 lines changes the FY outer count only from \(5\) to \(7\) and the runtime from \(732.1\) to \(882.5\) s, whereas both KKT/MILP runs reach the 3600-s limit. This supports the expected structural advantage.


\subsection{Response Mechanism and Economic Interpretation}

Congestion relief is not produced by price magnitude alone. On IEEE 13, FY increases flexible-load shifting from 2.971 to 3.951 MWh and reduces total import by 12.48\%, while battery throughput changes by less than 0.01\%. On IEEE 34, flexible shifting rises by 51.79\% and import falls by 11.36\%, again with essentially unchanged throughput. Renewable utilization remains 100\% throughout. Thus the dominant mechanism is temporal redistribution of net withdrawal---through flexible demand and the timing of existing storage/local generation---rather than renewable curtailment or additional battery cycling.

The IEEE 34 KKT and FY rows have the same aggregate import, 26.77 MWh, but FY leaves 71.1\% less congestion than the KKT incumbent. Likewise, Central eliminates IEEE 13 congestion while importing more energy than FY. These comparisons demonstrate that congestion depends on the temporal and electrical placement of withdrawals, not merely their horizon total. The common price has zero cross-hub spread, so spatial differentiation is achieved only through heterogeneous local assets, tariffs, and sensitivities; locational adders could further close the centralized gap but would change the market instrument being tested.

The FY residuals of \(3.41\times10^{-13}\) and \(1.71\times10^{-13}\) verify consistency of the reported schedules with the corresponding fixed-price convex follower models. The outer algorithm terminates when no improving local candidate is found, consistent with the stationary-point guarantees in Section~\ref{subsec:convergence}.



Additionally, one should note that the computational comparison here is structural, not a claim that FY must dominate every KKT implementation. To clarify, let \(n_x\) denote the follower-variable dimension, \(m_I\) and \(m_E\) the numbers of follower inequalities and equalities, and \(p\) the leader-price dimension. A conventional KKT/MPEC route augments the primal model with \(m_I+m_E\) dual variables and \(m_I\) complementarity relations; a disjunctive reformulation can additionally require up to \(m_I\) binary constructs. The FY route retains the primal response and the \(p\)-dimensional price, replacing constraint-wise complementarity with a conjugate/epigraph evaluation and a scalar residual per follower. {\color{black}Table~C.2 summarizes what is and is not, certified by the two routes.

Finally, four extra synthetic stress tests are reported in the online appendix. Their main message is that FY dynamic price coordination can substantially reduce overloads and again acts better than KKT; however, residual congestion becomes physical when network limits and available flexibility are binding.}



\section{Conclusion}
\label{sec:conclusion}

This paper developed an FY convex-analytic framework for bilevel dynamic congestion pricing in active distribution networks. By embedding follower feasibility into an extended convex objective, the lower-level price response is represented through a nonnegative FY optimality gap, avoiding explicit KKT complementarity constraints and big-\(M\) disjunctions. The remaining bilinear price-response term is handled by a DC decomposition and an SCA algorithm, yielding continuous convex subproblems with residual-based certificates of lower-level consistency. The framework is tailored to DSO coordination of price-responsive hubs/VPPs with flexible demand, storage, renewable curtailment, dispatchable generation, and PTDF-type congestion monitoring. In the IEEE 13- and 34-node studies, the standard feeder topology and load geography are preserved while controllable aggregators and dynamic price adders are introduced transparently. The proposed FY approach reduced total overload by \(96.89\%\) and \(96.45\%\), maintained FY residuals at numerical precision, and produced lower residual congestion within the imposed computational budget in cases where the complementarity-based KKT/MILP benchmark returned time-limited incumbents. The results show that congestion relief is driven primarily by the timing and network location of price-induced withdrawals, not by aggregate import reduction or curtailment alone. Future work will combine the reformulation with unbalanced convex AC models, uncertainty-aware response learning~\cite{cianchi2026learning}, distributed SCA, and settlement designs with revenue-neutrality and participant-surplus constraints.




\ifextendedappendices

\appendix[Supplementary Analyses and Stress Tests]

\section{Appendix}
\label{app:main}






\counterwithin{equation}{subsection}
\counterwithin{table}{subsection}
\counterwithin{figure}{subsection}

\renewcommand{\theequation}{\Alph{subsection}.\arabic{equation}}
\renewcommand{\thetable}{\Alph{subsection}.\arabic{table}}
\renewcommand{\thefigure}{\Alph{subsection}.\arabic{figure}}

\makeatletter
\@ifpackageloaded{hyperref}{%
  \renewcommand{\theHequation}{app.\Alph{subsection}.\arabic{equation}}%
  \renewcommand{\theHtable}{app.\Alph{subsection}.\arabic{table}}%
  \renewcommand{\theHfigure}{app.\Alph{subsection}.\arabic{figure}}%
}{}
\makeatother




This appendix collects material supporting the theoretical and numerical analysis in the main paper. It provides proof details, feeder-processing and reproducibility information, additional structural comparisons, and four synthetic stress tests designed to isolate the effects of congestion severity, renewable/storage arbitrage, and synchronized EV charging.

\subsection{Proof Details for the Fenchel--Young Reformulation}
\label{app:proofs}
Let
\(\phi=\psi+I_{\X}\), where \(\psi\) is proper, closed, and convex and \(\X\) is nonempty, closed, and convex.  For a fixed price coefficient \(\lambda\), the follower problem is
\begin{equation}
\min_{z\in\R^n}\{\phi(z)+\lambda^\top z\}.
\end{equation}
The point \(x\) is optimal if and only if
\begin{equation}
0\in \partial\phi(x)+\lambda,
\quad\text{or equivalently}\quad -\lambda\in\partial\phi(x).
\end{equation}
By the equality case of the FY inequality,
\begin{equation}
    -\lambda\in\partial\phi(x)
    \quad\Longleftrightarrow\quad
    \phi(x)+\phi^*(-\lambda)+\lambda^\top x=0.
\end{equation}
The left-hand side is nonnegative for all \((x,\lambda)\). Hence the lower-level feasible-response condition can be represented without explicit dual variables or complementarity products by the zero level set of the FY gap.

For the SCA step, write
\begin{equation}
D_\phi(x,-\lambda)=h(x,\lambda)-o(x,\lambda),
\end{equation}
where
\begin{align}
 h(x,\lambda)&=\phi(x)+\phi^*(-\lambda)+\frac14\|x+\lambda\|^2,\\
 o(x,\lambda)&=\frac14\|x-\lambda\|^2 .
\end{align}
Both \(h\) and \(o\) are convex, and the identity follows from
\(\lambda^\top x=\frac14\|x+\lambda\|^2-\frac14\|x-\lambda\|^2\).  Because \(o\) is convex,
\begin{equation}
    o(x,\lambda)\ge o(x^k,\lambda^k)+\langle \nabla o(x^k,\lambda^k),\col(x-x^k,\lambda-\lambda^k)\rangle .
\end{equation}
Therefore \(-o\) is upper-bounded by the negative of this affine minorant, yielding the majorized gap used in Algorithm 1.  This is the sign-sensitive step: using the opposite sign would minorize rather than majorize the penalized objective and would destroy the standard descent interpretation of the convex--concave procedure.

\subsection{Computational Workflow}
The computational workflow is summarized in Fig.~\ref{fig:fy_sca_pipeline}. 
\begin{figure}[!htb]
\centering
\begin{tikzpicture}[
 font=\scriptsize,
 block/.style={draw,rounded corners,align=center,text width=0.82\columnwidth,minimum height=5.4mm,fill=gray!6,inner sep=1.7pt},
 arr/.style={-{Latex[length=1.45mm]},line width=0.65pt}
]
\node[block] (a) {Affine price-entry follower: \(\phi_i(x_i)+\lambda_i(\delta_i)^\top x_i\)};
\node[block,below=2.5mm of a] (b) {FY gap: \(D_{\phi_i}(x_i,-\lambda_i(\delta_i))\)};
\node[block,below=2.5mm of b] (c) {DC split and affine decomposition: \(h_i-o_i\mapsto\widehat D_i^k\)};
\node[block,below=2.5mm of c] (d) {Solve convex SCA subproblem; evaluate \(r_{\max}^{k+1}\) and \(s^{k+1}\)};
\node[block,below=2.5mm of d] (e) {Stop, or update \(\rho_k\) only when the FY residual stalls};
\draw[arr] (a)--(b); \draw[arr] (b)--(c); \draw[arr] (c)--(d); \draw[arr] (d)--(e);
\draw[arr] (e.east) -- ++(0.25,0) |- (b.east);
\end{tikzpicture}
\vspace{-0.5ex}
\caption{Fenchel--Young SCA workflow. The feedback path increases the penalty only when lower-level consistency ceases to improve.}
\label{fig:fy_sca_pipeline}
\end{figure}
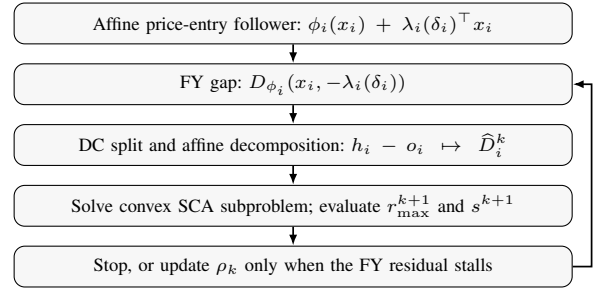

\subsection{IEEE Feeder Processing, Augmentation, and Full Diagnostics}
\label{app:real_feeder_details}
The active-hub sets are
\begin{align*}
\Hubs_{13}&=\{634,645,646,671,675,692,611,652\},\\
\Hubs_{34}&=\{890,844,\mathrm{mid860},\mathrm{mid822},\mathrm{mid836},848,860,830\}.
\end{align*}
These labels are benchmark bus identifiers; ``mid'' denotes retained intermediate nodes in the processed OpenDSS graph. If \(P_b^0\) is the original static real load at bus \(b\), the controllable and passive shares are
\begin{equation}
 P_{b}^{\rm ctrl}=\zeta P_b^0,\qquad P_b^{\rm pass}=(1-\zeta)P_b^0,\qquad \zeta=0.85.
 \label{eq:app_load_split}
\end{equation}
The parameter \(\zeta\) is the controllable-load share. Each controllable share is assigned to the electrically nearest active hub; passive load remains at its original bus and forms \(f^{\rm bg}_{\ell t}\).

Study branch ratings are
\begin{equation}
 \bar f_\ell=\max\{F_\ell^{\min},\kappa^{\rm stress}\max_t|f^{\rm base}_{\ell t}|\},\qquad \kappa^{\rm stress}=0.82,
 \label{eq:app_capacity_rule}
\end{equation}
where \(f^{\rm base}_{\ell t}\) is no-price base flow, \(F_\ell^{\min}\) is a phase-dependent floor, and \(\kappa^{\rm stress}\) is the stress factor. Source and transformer elements use a protected envelope exceeding their base-flow maximum. 

The deterministic feeder load multiplier is
\begin{align}
 \omega_t=\omega^0
 &+a^{\rm m}\exp\!\left[-\frac{(t-t^{\rm m})^2}{2\sigma_{\rm m}^2}\right]\notag\\
 &+a^{\rm e}\exp\!\left[-\frac{(t-t^{\rm e})^2}{2\sigma_{\rm e}^2}\right]\notag\\
 &-a^{\rm d}\exp\!\left[-\frac{(t-t^{\rm d})^2}{2\sigma_{\rm d}^2}\right].
 \label{eq:app_feeder_load_shape}
\end{align}
normalized to unit mean. Here \(\omega^0\) is the base multiplier; \(a^{\rm m},a^{\rm e},a^{\rm d}\) are morning, evening, and midday-dip amplitudes; \(t^{\rm m},t^{\rm e},t^{\rm d}\) locate those components; and \(\sigma_{\rm m},\sigma_{\rm e},\sigma_{\rm d}\) determine their widths. Available PV follows \([\sin(\pi(t-6)/12)]_+^{1.7}\). The exogenous buy price is \(42+28g_t^{\rm e}+8g_t^{\rm m}-16g_t^{\rm pv}\) EUR/MWh, clipped below at 4 EUR/MWh, where the three \(g\)-terms are the normalized evening, morning, and solar shapes. These profiles are study inputs, separate from feeder provenance. Tables~\ref{tab:ieee_case_setup}--\ref{tab:route_scope} detail the IEEE cases' augmentations and computational complexity diagnostics, respectively.
\begin{table*}[!t]
\centering
\caption{IEEE Feeder Construction, Study Augmentations, and Reproducibility Settings}
\label{tab:ieee_case_setup}
\scriptsize
\setlength{\tabcolsep}{2.6pt}
\renewcommand{\arraystretch}{1.02}
\begin{tabular}{L{0.18\textwidth}L{0.38\textwidth}L{0.38\textwidth}}
\toprule
Item & IEEE 13-node case & IEEE 34-node case\\
\midrule

Benchmark feeder
&
Standard IEEE 13-node radial distribution test feeder \cite{kersting2001radial}. The original feeder topology and static spot-load locations are preserved.
&
Standard IEEE 34-node radial distribution test feeder \cite{kersting2001radial}. The original feeder topology and static spot-load locations are preserved.
\\

Nominal voltage
&
\(4.16\) kV.
&
\(24.9\) kV.
\\

Processed network used for pricing
&
\(16\) processed nodes and \(18\) monitored branches after including the source/DSO representation used to construct the PTDF-type active-power sensitivity model.
&
\(56\) processed nodes and \(60\) monitored branches after retaining intermediate nodes required by the processed OpenDSS graph and the PTDF-type active-power sensitivity model.
\\

Preserved data
&
Original radial connectivity, bus identifiers, feeder geography, and static active-power load locations.
&
Original radial connectivity, bus identifiers, feeder geography, and static active-power load locations.
\\

Study augmentations
&
Controllable aggregator hubs, DER portfolios, time-varying load/PV/price profiles, and congestion-study branch limits are added for dynamic-pricing experiments.
&
Same augmentation structure: controllable aggregator hubs, DER portfolios, time-varying load/PV/price profiles, and congestion-study branch limits are added.
\\

Controllable hub buses
&
\(\{634,645,646,671,675,692,611,652\}\).
&
\(\{890,844,\mathrm{mid}860,\mathrm{mid}822,\mathrm{mid}836,848,860,830\}\).
\\

Hub-bus interpretation
&
Each listed bus is treated as a price-responsive aggregator/hub connection point. Hubs are connected to buses, not directly to lines; their net withdrawals affect monitored branch flows through \(H_{\ell i}\).
&
Same interpretation. The labels \(\mathrm{mid}860\), \(\mathrm{mid}822\), and \(\mathrm{mid}836\) denote retained intermediate graph nodes in the processed feeder model.
\\

Hub-selection rule
&
Loaded or electrically influential buses are promoted to hubs so that distinct feeder zones influence the monitored congestion constraints.
&
Loaded or electrically influential buses, including retained intermediate graph nodes, are promoted to hubs so that upstream, mid-feeder, and downstream zones are represented.
\\

Original static load
&
Total original static active-power load is \(3.466\) MW.
&
Total original static active-power load is \(1.769\) MW.
\\

Load assignment to hubs
&
For every original static load \(P_b^0\), \(85\%\) is assigned to the electrically nearest active hub and \(15\%\) remains at bus \(b\) as passive background demand.
&
Same rule: \(85\%\) of each original static load is assigned to the nearest active hub and \(15\%\) remains as passive background demand.
\\

Background demand
&
The passive \(15\%\) load component contributes to \(f_{\ell t}^{\rm bg}\) in the linear sensitivity model and is not optimized by the hub response problem.
&
Same background-demand treatment.
\\

Controllable hub-load notation
&
\(P_i^{\rm ctrl}\) denotes the total active-power load assigned to hub \(i\) before time-profile scaling, i.e., the sum of the \(85\%\) controllable portions assigned to hub \(i\).
&
Same definition of \(P_i^{\rm ctrl}\).
\\

Fixed/flexible split
&
At each hub, \(74\%\) of the assigned profile is fixed demand and \(26\%\) is flexible demand, except every third hub in the listed order, where the split is \(70\%/30\%\).
&
Same fixed/flexible split rule.
\\

Flexible-load bounds
&
Flexible demand can move within \(\pm35\%\) of its preferred profile, except every third hub in the listed order, where the bound is \(\pm25\%\). Total flexible energy is conserved over the horizon.
&
Same flexible-load bounds and energy-conservation rule.
\\

PV capacity
&
Nominal PV capacity at hub \(i\) is
\(\max\{0.02,\,0.45P_i^{\rm ctrl}\}\) MW.
&
Nominal PV capacity at hub \(i\) is
\(\max\{0.02,\,0.48P_i^{\rm ctrl}\}\) MW.
\\

Battery capacity
&
Battery power capacity is \(\max\{0.025,\,0.32P_i^{\rm ctrl}\}\) MW. Energy capacity is twice the battery power capacity, and the initial state of charge is \(50\%\).
&
Battery power capacity is \(\max\{0.025,\,0.34P_i^{\rm ctrl}\}\) MW. Energy capacity is twice the battery power capacity, and the initial state of charge is \(50\%\).
\\

Dispatchable DER
&
Dispatchable capacity is \(\max\{0,\,0.18P_i^{\rm ctrl}\}\) MW, with a scarcity-shaped availability profile and hub-dependent marginal cost.
&
Dispatchable capacity is \(\max\{0,\,0.17P_i^{\rm ctrl}\}\) MW, with the same availability construction.
\\

Import/export limits
&
Import limit is
\(\max\{0.10,\,1.55P_i^{\rm ctrl}\omega^{\max}+0.08\}\) MW, and export limit is
\(\max\{0.04,\,0.50P_i^{\rm ctrl}+0.03\}\) MW.
&
Same import/export limit formulas.
\\

Profile-scaling notation
&
\(\omega^{\max}\) is the maximum value of the normalized 24-period demand profile used to scale hub demand.
&
Same definition of \(\omega^{\max}\).
\\

Congestion-study branch limits
&
Branch limits are:
\[
\bar f_\ell=\max\{F_\ell^{\rm floor},\,0.82\max_{t\in\Time}|f_{\ell t}^{\rm base}|\}.
\]
Source/transformer branches are protected by higher limits when needed to avoid artificial bottlenecks at the feeder head.
&
Same construction rule with the same stress factor \(0.82\).
\\

Branch-limit notation
&
\(f_{\ell t}^{\rm base}\) is the no-pricing base-case flow on monitored line \(\ell\) at time \(t\), and \(F_\ell^{\rm floor}\) is a line-specific minimum study limit used to avoid unrealistically small branch limits.
&
Same definitions of \(f_{\ell t}^{\rm base}\) and \(F_\ell^{\rm floor}\).
\\

Substation limit
&
\(\bar S\) is set to \(0.92\) times the peak no-pricing total import. The resulting value is \(3.889\) MW.
&
\(\bar S\) is set to \(0.92\) times the peak no-pricing total import. The resulting value is \(2.031\) MW.
\\

Dynamic-price policy
&
Common broadcast adder \(\delta_{it}=\delta_t\) for all hubs, with
\(-40\le\delta_t\le80\) EUR/MWh. The buy/sell treatment follows \eqref{eq:effective_prices}; the reported IEEE cases use a symmetric net-withdrawal adder, i.e., \(\chi^{\rm sell}=1\).
&
Same dynamic-price policy.
\\

Time horizon
&
\(T=24\) hourly settlement periods with \(\Delta t=1\) h.
&
\(T=24\) hourly settlement periods with \(\Delta t=1\) h.
\\

Benchmark methods
&
Base case, centralized full-information dispatch, complementarity-based KKT/MILP benchmark, and Quad-FY.
&
Base case, centralized full-information dispatch, complementarity-based KKT/MILP benchmark, and Quad-FY.
\\

Solver environment
&
Intel Core i9-13900K CPU, \(16.0\) GB RAM, 64-bit Windows 10.
&
Same solver environment.
\\

\bottomrule
\end{tabular}
\end{table*}

\begin{table}[!htb]
\centering
\caption{Interpretation of the Two Single-Level Routes}
\label{tab:route_scope}
\scriptsize
\setlength{\tabcolsep}{2.6pt}
\renewcommand{\arraystretch}{.95}
\begin{tabular}{L{.29\columnwidth}L{.31\columnwidth}L{.31\columnwidth}}
\toprule
Item & KKT/MILP benchmark & Proposed FY-SCA \\
\midrule
Added structure & Duals, complementarity, and possible binaries & Conjugate epigraph/oracle; no complementarity \\
Subproblem class & MILP/MPEC after linearization & Continuous convex program per SCA iteration \\
Reported guarantee & Incumbent unless the global solve closes & Fixed-price follower consistency plus stationarity of the penalized outer model \\
Main scaling driver & Number of follower constraints & Response/price dimension and conjugate evaluation \\
\bottomrule
\end{tabular}
\end{table}



\subsection{Synthetic Cases Network Model and Profiles}
\label{appsubsec:network_model}

Fig.~\ref{fig:network_tikz} shows the four-bus, three-hub feeder used in the computational study of synthetic cases. The topology has one DSO/slack bus, three heterogeneous hubs, and six monitored distribution lines. The hub connections are physical terminals rather than bilateral commercial trades; hence they are drawn without arrowheads, and the sign of the net withdrawal determines whether a hub is importing from or exporting to the feeder. Detailed definitions and specifications of the network variables are provided in the Nomenclature.

\begin{figure}[htb]
\centering
\resizebox{0.82\columnwidth}{!}{%
\begin{tikzpicture}[
    font=\scriptsize,
    bus/.style={circle, draw, minimum size=7mm, fill=gray!7},
    hub/.style={rectangle, rounded corners, draw, align=center, minimum width=14mm, minimum height=7mm, fill=gray!4},
    line/.style={thick},
    hubline/.style={thick, dashed}
]
\node[bus] (dso) at (0,0) {DSO};
\node[bus] (b1) at (-1.75,1.22) {1};
\node[bus] (b2) at (-1.75,-1.22) {2};
\node[bus] (b3) at (1.95,0) {3};
\node[hub] (h1) at (-3.30,1.22) {Hub 1};
\node[hub] (h2) at (-3.30,-1.22) {Hub 2};
\node[hub] (h3) at (3.50,0) {Hub 3};
\draw[line] (dso) -- node[above] {$\ell_1$} (b1);
\draw[line] (dso) -- node[below] {$\ell_2$} (b2);
\draw[line] (dso) -- node[above] {$\ell_3$} (b3);
\draw[line] (b1) -- node[left] {$\ell_4$} (b2);
\draw[line] (b1) -- node[above] {$\ell_5$} (b3);
\draw[line] (b2) -- node[below] {$\ell_6$} (b3);
\draw[hubline] (h1) -- node[above] {$w_{1t}$} (b1);
\draw[hubline] (h2) -- node[below] {$w_{2t}$} (b2);
\draw[hubline] (h3) -- node[above] {$w_{3t}$} (b3);
\end{tikzpicture}}
\caption{Three-hub active distribution network used in the study of synthetic cases. Hub-terminal lines are not directed arcs; positive $w_{it}$ denotes withdrawal and negative $w_{it}$ denotes export.}
\label{fig:network_tikz}
\end{figure}
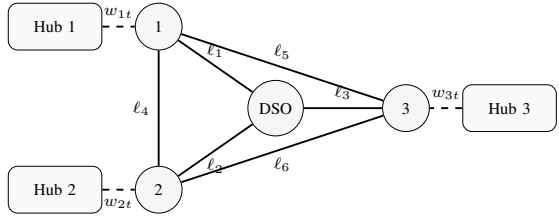

\begin{table*}[htb]
\centering
\caption{Dynamic Test Cases and Common Computational Settings}
\label{tab:case_design}
\scriptsize
\setlength{\tabcolsep}{3.0pt}
\renewcommand{\arraystretch}{0.92}
\begin{tabular}{lL{0.60\textwidth}cccccc}
\toprule
Case & Stress narrative & $\bar S^{\pm}$ & Cap. scale & RE scale & Gen. scale & Arbitrage\\
 & & & (MW) & & & & \\
\midrule
A & Balanced reference feeder with moderate congestion and neutral prices. & 2.55 & 1.00 & 1.00 & 1.00 & No\\
B & Urban evening congestion with tighter line limits and higher background demand. & 2.40 & 0.78 & 0.88 & 1.08 & No\\
C & Renewable/storage arbitrage with larger PV and storage and scarcity spreads. & 2.45 & 0.95 & 1.45 & 0.90 & Yes\\
R1 & Synchronized EV-charging event with a short severe import pulse. & 3.15 & 1.45 & 0.92 & 1.02 & No\\
\midrule
\multicolumn{8}{p{0.96\textwidth}}{\emph{Common settings:} 3 hubs, 4 buses, 6 monitored lines, $T=12$, $\dd=1$ h, $\pm75\%$ flexible-load range by hub, and each hub has batteries, PV, and dispatchable generation.}\\
\multicolumn{8}{p{0.96\textwidth}}{\emph{Price policy and computation:} one shared $\lambda_t$ for all hubs, seed 42, SciPy/HiGHS LP/MILP, SciPy trust-constr QP.}\\
\multicolumn{8}{p{0.96\textwidth}}{\emph{Hardware:} Intel Core i9-13900K CPU, 16.0 GB RAM, 64-bit Windows 10.}\\
\bottomrule
\end{tabular}
\end{table*}


The four cases in Table~\ref{tab:case_design} are controlled synthetic stress tests generated with a fixed seed. They are not empirically calibrated to a named feeder; instead, they are designed to isolate representative operating regimes through transparent combinations of diurnal load peaks, time-varying price components, daylight-driven PV availability, storage-arbitrage opportunities, and rare synchronized demand pulses. Although synthetic, the profiles deliberately follow motifs commonly observed in public power-system data products, including hourly load and price series, transparency-platform data, U.S. balancing-authority load shapes, and PV production profiles consistent with public PVWatts-style models \cite{hirth2018entsoe,eiaGridMonitor,nrelPVWatts}.




\begin{figure}[htb]
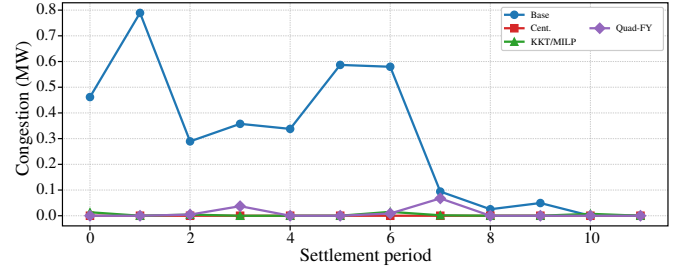
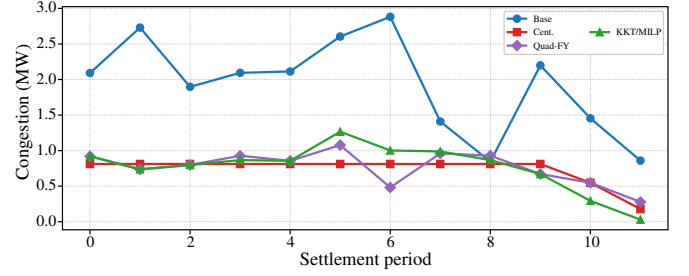
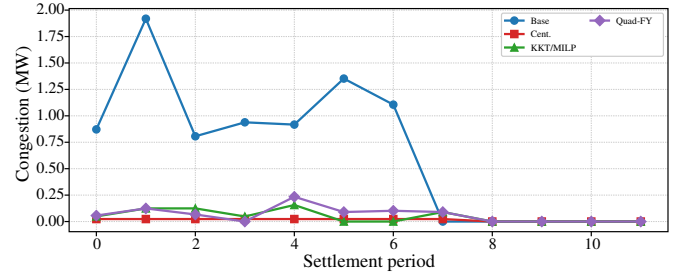
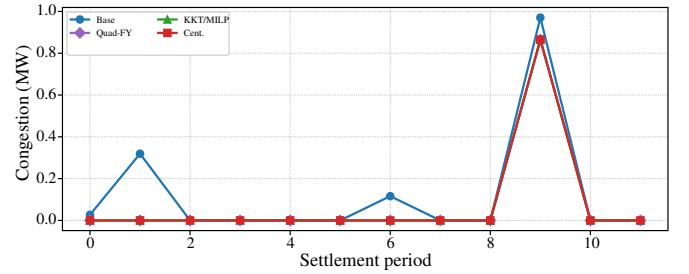

\centering
\subfloat[\CongCap{Case A: balanced reference.}\label{fig:cong-a-v}]{%
  \IEEECongA[\columnwidth]}\\[-0.25ex]
\subfloat[\CongCap{Case B: urban congestion stress.}\label{fig:cong-b-v}]{%
  \IEEECongB[\columnwidth]}\\[-0.25ex]
\subfloat[\CongCap{Case C: renewable/storage arbitrage.}\label{fig:cong-c-v}]{%
  \IEEECongC[\columnwidth]}\\[-0.25ex]
\subfloat[\CongCap{Case R1: synchronized EV peak.}\label{fig:cong-r1-v}]{%
  \IEEECongRone[\columnwidth]}
\caption{Period-level congestion trajectories for the four scenarios.}
\label{fig:app_trajectories}
\end{figure}


For hub \(i\) and settlement period \(t\), a representative fixed-load profile is generated as
\begin{align}
D^{\rm fix}_{it}
=
\kappa_i D_i^{0}
\Bigg[
1
&+
a_i^{\rm m}
\exp\!\left(
-\frac{(t-t_i^{\rm m})^2}{2\sigma_{i,\rm m}^2}
\right)
\notag\\
&+
a_i^{\rm e}
\exp\!\left(
-\frac{(t-t_i^{\rm e})^2}{2\sigma_{i,\rm e}^2}
\right)
\Bigg]
+
\xi_{it}.
\label{eq:app_fixed_load_profile}
\end{align}
Here, \(D^{\rm fix}_{it}\) is the inflexible demand of hub \(i\) at period \(t\); \(D_i^{0}\) is the nominal fixed-demand level of hub \(i\); \(\kappa_i\) is the case-dependent demand multiplier; \(a_i^{\rm m}\) and \(a_i^{\rm e}\) are the relative amplitudes of the morning and evening demand components; \(t_i^{\rm m}\) and \(t_i^{\rm e}\) are their peak periods; \(\sigma_{i,\rm m}\) and \(\sigma_{i,\rm e}\) control the width of the two peaks; and \(\xi_{it}\) is a small deterministic perturbation used to avoid perfectly symmetric profiles. The morning and evening terms encode residential, campus, and logistics-driven consumption patterns.

The available PV production profile is represented by a daylight bell shape,
\begin{equation}
PV^{\rm avail}_{it}
=
\kappa_i^{\rm pv}
\bar P_i^{\rm pv}
\left[
\sin\!\left(
\pi\frac{t-t^{\rm sr}}{t^{\rm ss}-t^{\rm sr}}
\right)
\right]_+^{\gamma_i}.
\label{eq:app_pv_profile}
\end{equation}
where \(PV^{\rm avail}_{it}\) is the maximum available PV generation at hub \(i\) and period \(t\); \(\bar P_i^{\rm pv}\) is the nominal PV capacity of hub \(i\); \(\kappa_i^{\rm pv}\) is the case-dependent PV multiplier; \(t^{\rm sr}\) and \(t^{\rm ss}\) are sunrise-like and sunset-like period indices within the optimization horizon; \(\gamma_i\) controls the sharpness of the daylight production curve; and \([y]_+=\max\{y,0\}\) denotes the positive-part operator. This construction ensures zero PV availability outside the daylight window and a smooth midday production peak.

The market component of the retail price is constructed from off-peak, midday, and evening-scarcity components,
\begin{align}
\pi^{\rm mkt}_{t}
=
\pi^{0}
&+
\pi^{\rm mid}
\exp\!\left(
-\frac{(t-t^{\rm mid})^2}{2\sigma_{\rm mid}^2}
\right)
\notag\\
&+
\pi^{\rm eve}
\exp\!\left(
-\frac{(t-t^{\rm eve})^2}{2\sigma_{\rm eve}^2}
\right).
\label{eq:app_market_price_profile}
\end{align}
Here, \(\pi^{\rm mkt}_{t}\) is the exogenous market-price component at period \(t\); \(\pi^{0}\) is the base off-peak price level; \(\pi^{\rm mid}\) and \(\pi^{\rm eve}\) are the amplitudes of the midday and evening price components; \(t^{\rm mid}\) and \(t^{\rm eve}\) locate these price components in the horizon; and \(\sigma_{\rm mid}\) and \(\sigma_{\rm eve}\) determine their temporal widths. The dynamic congestion adder \(\delta\) is then optimized on top of this exogenous market component in the upper-level problem.

Rare-event stress is added through a horizon-relative pulse,
\begin{equation}
P^{\rm event}_{it}
=
A_i
\exp\!\left[
-\frac{(t-t^{\rm ev})^2}{2\sigma_{\rm ev}^2}
\right].
\label{eq:app_event_profile}
\end{equation}
where \(P^{\rm event}_{it}\) is the additional event-driven demand imposed on hub \(i\) at period \(t\); \(A_i\) is the hub-specific event amplitude; \(t^{\rm ev}\) is the event center; and \(\sigma_{\rm ev}\) controls the event duration. This term is used in Case R1 to represent synchronized EV charging.

The case multipliers are those reported in Table~I of the main text. Case A is the balanced reference case. Case B increases background demand and tightens network limits to create an urban congestion-stress condition. Case C increases PV and storage availability to test renewable-driven flexibility and arbitrage. Case R1 superimposes the rare-event pulse in \eqref{eq:app_event_profile} to test whether the pricing mechanism can handle a concentrated evening import shock.

The benchmarks are: no dynamic pricing (baseline), centralized linear dispatch (full-information lower bound), exact linear bilevel KKT/MILP, and the proposed quadratic FY coordination. The centralized case is not an implementable market mechanism; it is a lower bound on physically achievable congestion. The KKT/MILP row is exact only when the solver closes the relevant gap, or otherwise should be read with its reported status. The FY row solves convex fixed-price hub QPs and certifies follower consistency.


\subsection{Congestion Mitigation Results Across the Four Synthetic Profiles}
Table~\ref{tab:app_full_benchmarks} reports the main physical and computational outcomes. Our dynamic pricing reduces total congestion by $96.70\%$, $60.36\%$, $90.35\%$, and $39.62\%$ in Cases A, B, C, and R1, respectively. Fig.~\ref{fig:app_trajectories} reports the congestion trajectories generated by the computations for each case. The reduction is strongest in Cases A and C because congestion is largely caused by shiftable net withdrawals that can be retimed through the common price signal. In Case B, the feeder is much tighter and the centralized lower bound still has 8.849 MW of residual overload; hence the FY solution is only 0.340 MW above the centralized lower bound and slightly better than the KKT/MILP incumbent in total congestion. In R1, all non-baseline methods reach the same residual congestion as the centralized lower bound, indicating that the remaining overload is physical and topology/resource limited rather than a price-design failure.


\begin{table*}[htb]
\centering
\caption{Congestion and Runtime Outcomes Across Cases}
\label{tab:app_full_benchmarks}
\footnotesize
\setlength{\tabcolsep}{3.0pt}
\renewcommand{\arraystretch}{0.94}
\begin{tabular}{llrrrrrrL{0.25\textwidth}}
\toprule
Case & Method & Total & Line & Sub. & Red. & Time & Avg. adder & Status/certificate \\
 & & (MW) & (MW) & (MW) & (\%) & (s) & (EUR/MWh) & \\
\midrule
A & Baseline & 3.571 & 2.290 & 1.281 & -- & 0.092 & -- & LP/QP optimal \\
A & Central & 0.000 & 0.000 & 0.000 & 100.00 & 0.123 & 0.00 & lower bound \\
A & KKT & 0.042 & 0.042 & 0.000 & 98.81 & 189.1 & 33.38 & auto runtime limit reached \\
A & FY & 0.118 & 0.118 & 0.000 & 96.70 & 69.99 & 32.63 & FY residual $1.14\times10^{-13}$ \\
\midrule
B & Baseline & 23.180 & 16.780 & 6.402 & -- & 0.093 & -- & LP/QP optimal \\
B & Central & 8.849 & 8.849 & 0.000 & 61.83 & 0.130 & 0.00 & lower bound \\
B & KKT & 9.299 & 8.944 & 0.356 & 59.88 & 300.2 & 40.59 & exact KKT/MILP incumbent \\
B & FY & 9.189 & 8.947 & 0.241 & 60.36 & 90.92 & 41.78 & FY residual $2.27\times10^{-13}$ \\
\midrule
C & Baseline & 7.911 & 4.528 & 3.382 & -- & 0.094 & -- & LP/QP optimal \\
C & Central & 0.189 & 0.189 & 0.000 & 97.62 & 0.125 & 0.00 & lower bound \\
C & KKT & 0.592 & 0.592 & 0.000 & 92.52 & 321.1 & 35.18 & price-change tolerance reached \\
C & FY & 0.764 & 0.735 & 0.029 & 90.35 & 67.98 & 35.21 & FY residual $1.71\times10^{-13}$ \\
\midrule
R1 & Baseline & 1.433 & 0.971 & 0.462 & -- & 0.093 & -- & LP/QP optimal \\
R1 & Central & 0.865 & 0.865 & 0.000 & 39.62 & 0.123 & 0.00 & lower bound \\
R1 & KKT & 0.865 & 0.865 & 0.000 & 39.62 & 311.6 & 11.34 & price-change tolerance reached \\
R1 & FY & 0.865 & 0.865 & 0.000 & 39.62 & 100.4 & 20.12 & FY residual $1.14\times10^{-13}$ \\
\bottomrule
\end{tabular}
\end{table*}

The mechanism behind the reduction is visible from the import and flexibility metrics, and solver diagnostics in Tables~\ref{tab:app_response_metrics}--\ref{tab:app_solver_diagnostics}. The price signal does not create or destroy demand: fixed load is served and flexible-load energy is conserved. Instead, the DSO price increases the marginal cost of importing during overloaded periods, causing hubs to shift flexible consumption, modify battery schedules, use dispatchable generation when economical, and reduce feeder-head import. The total import falls by 15.0\%, 22.3\%, 12.5\%, and 15.0\% from the baseline in Cases A, B, C, and R1. Case B therefore improves mainly through import reduction and modest flexible shifting under tight line limits. Case C uses the largest flexible shift, consistent with its renewable/storage-arbitrage construction. In R1 the baseline already activates substantial flexibility, so the additional benefit comes from synchronizing the event-period price with the physical congestion pulse; once the centralized residual is reached, higher price effort cannot remove the remaining overload.

\begin{table*}[htb]
\centering
\caption{Detailed Price-Response Metrics for the Implementable Methods}
\label{tab:app_response_metrics}
\footnotesize
\setlength{\tabcolsep}{3.0pt}
\renewcommand{\arraystretch}{0.94}
\begin{tabular}{llrrrrrrr}
\toprule
Case & Method & Avg. adder & Ramp & Flex shift & Battery throughput & Renewable util. & Import\\
 & & (EUR/MWh) & (EUR/MWh) &  (MWh) & (MWh) & (\%) & (MWh)\\
\midrule
A & Baseline & -- & --  & 0.164 & 1.340 & 100 & 30.45\\
A & KKT & 33.38 & 19.39 & 1.985 & 0.521 & 100 & 26.15\\
A & FY & 32.63 & 13.49  & 2.001 & 0.521 & 100 & 25.89\\
\midrule
B & Baseline & -- & --  & 0.000 & 1.495 & 100 & 35.17\\
B & KKT & 40.59 & 8.286  & 0.000 & 0.455 & 100 & 27.21\\
B & FY & 41.78 & 9.609  & 0.998 & 0.455 & 100 & 27.33\\
\midrule
C & Baseline & -- & --  & 0.294 & 2.359 & 100 & 27.42\\
C & KKT & 35.18 & 23.81  & 2.865 & 1.245 & 100 & 23.98\\
C & FY & 35.21 & 17.75  & 2.865 & 1.244 & 100 & 23.98\\
\midrule
R1 & Baseline & -- & -- & 2.446 & 2.047 & 100 & 34.88\\
R1 & KKT & 11.34 & 30.00 & 2.257 & 1.909 & 100 & 32.05\\
R1 & FY & 20.12 & 0.391 & 2.446 & 2.085 & 100 & 29.65\\
\bottomrule
\end{tabular}
\end{table*}

\begin{table*}[htb]
\centering
\caption{Solver and Certificate Diagnostics}
\label{tab:app_solver_diagnostics}
\footnotesize
\setlength{\tabcolsep}{3.0pt}
\renewcommand{\arraystretch}{0.94}
\begin{tabular}{lrrrrL{0.55\textwidth}}
\toprule
Case & FY iter. & FY time & FY residual  & KKT time & Interpretation\\
 & & (s) &   & (s) & \\
\midrule
A & 11 & 69.99 & $1.14\times10^{-13}$ &  189.1 & KKT is $2.70\times$ slower and reaches a runtime status; FY certifies the same lower-level response.\\
B & 13 & 90.92 & $2.27\times10^{-13}$ &  300.2 & The tight network drives residual physical congestion; FY improves the KKT incumbent in total congestion.\\
C & 9 & 67.98 & $1.71\times10^{-13}$ &  321.1 & Storage and PV create the largest flex response; KKT has a larger ramp and longer runtime.\\
R1 & 3 & 100.4 & $1.14\times10^{-13}$ &  311.6 & The event is concentrated; FY reaches the centralized residual with a smoother price profile.\\
\bottomrule
\end{tabular}
\end{table*}

\subsection{Computational and Economic Interpretation of the Four Synthetic Cases Results}

Three observations are most important. First, FY should be assessed against both no-pricing and centralized lower-bound dispatch. Case B appears difficult because its baseline congestion is 23.18 MW, but the centralized lower bound is already 8.849 MW; hence the FY gap to the physical lower bound is only 0.340 MW, and FY slightly improves the reported KKT/MILP incumbent in total congestion. Second, congestion reduction is caused by the price-induced redistribution of net withdrawals, not by price magnitude alone. Case R1 reaches the centralized residual with a smaller average adder than Cases A--C because the stress is event-like; Case B needs the largest adder because scarcity is persistent, yet physical feeder limits dominate the remaining overload. Third, the computational burden shifts away from complementarity. KKT/MILP introduces dual and complementarity variables for lower-level inequalities and often stops at runtime or tolerance limits; FY solves continuous fixed-price hub QPs, reports residuals below $2.3\times10^{-13}$, and is $3.4\times$ faster on average.

Additionally, in terms of the structural scalability comparison between the KKT-based approach and the proposed FY, let \(n_x\) denote the dimension of the variable follower, \(m_i\) the number of follower inequalities, \(m_e\) the number of equalities, and \(p\) the dimension of the leader-price.  A KKT/MPEC reformulation introduces \(m_i\) nonnegative dual variables, \(m_e\) equality dual variables, stationarity equations, and \(m_i\) complementarity products.  If complementarity is enforced by a mixed-integer linearization, the number of binary or disjunctive constructs scales with \(m_i\).  The FY representation uses the primal response and price variables together with an optimality-loss evaluation; it does not introduce one complementarity product per lower-level inequality.  Table~\ref{tab:app_complexity} summarizes the resulting structural difference.

\begin{table}[htb]
\centering
\caption{Structural Growth of the Two Single-Level Routes}
\label{tab:app_complexity}
\scriptsize
\setlength{\tabcolsep}{3.0pt}
\renewcommand{\arraystretch}{0.94}
\begin{tabular}{lcc}
\toprule
Quantity & KKT/MPEC route & FY-SCA route\\
\midrule
Primal response & \(n_x\) & \(n_x\)\\
Leader price & \(p\) & \(p\)\\
Inequality duals & \(m_i\) & --\\
Equality duals & \(m_e\) & --\\
Complementarity products & \(m_i\) & --\\
Binary/disjunctive terms & up to \(m_i\) & --\\
Convex subproblem type & LP/QP/MILP/MPEC & continuous convex program\\
Main bottleneck & complementarity scaling & conjugate/oracle evaluation\\
\bottomrule
\end{tabular}
\end{table}

The observed runtimes in Table~\ref{tab:app_solver_diagnostics} are consistent with this structural comparison: KKT/MILP is between \(2.70\times\) and \(4.72\times\) slower than FY across the four reported cases, and it terminates with runtime or tolerance statuses in three of the four cases.  This does not mean that FY is globally exact for arbitrary bilevel programs; rather, it shows that for the affine price-entry convex followers studied here, the FY route supplies a compact residual-certified alternative whose computational burden is less sensitive to the number of lower-level inequalities.

\bibliographystyle{IEEEtran}
\bibliography{references}
\end{document}